\documentclass[12pt,a4paper,reqno]{amsart}
\usepackage[left=3cm,top=2cm,right=3cm,bottom=2cm]{geometry}
\usepackage{graphicx}
\usepackage{tikz}
\usetikzlibrary{decorations.pathreplacing, positioning}

 \usepackage{mathptmx}

\usepackage{multirow}
\usepackage{amsmath,amssymb, amscd}
\usepackage{tikz}
\usepackage{tkz-berge}
\usetikzlibrary{decorations,arrows,shapes}
\usepackage[noend,noline,boxed]{algorithm2e}
\usepackage[normalem]{ulem} 
\usepackage{url}

\newcommand{\x}{\mathbf{v}} 

\newtheorem{theorem}{Theorem}

\newtheorem{lemma}{Lemma}
\newtheorem{corollary}{Corollary}

\newtheorem{proposition}{Proposition}
\newtheorem{remark}{Remark}

\begin{document}

\title[Characterizing graphs with second largest distance eigenvalue less than -1/2 ]{Characterizing graphs with second largest distance eigenvalue less than -1/2}

\keywords{Distance matrix. Chordal graph. Second largest distance eigenvalue.
\subjclass{05C50, 05C75}}

\author[M. Abd\'on]{Miriam Abd\'on}
\address{Universidade Federal Fluminense, Niter\' oi, RJ, Brazil}
\email{miriam\_abdon@id.uff.br}

\author[L. Markenzon]{Lilian Markenzon}
\address{Universidade Federal do Rio de Janeiro, Rio de Janeiro, RJ, Brazil}
\email{markenzon@nce.ufrj.br}
\author[C. Vinagre]{Cybele T. M. Vinagre}
\address{Universidade Federal Fluminense, Niter\' oi, RJ, Brazil}
\email{cybele\_vinagre@id.uff.br}

\begin{abstract}
  Let $G$ be a connected graph with vertex set $V$. The distance, $d_G(u, v)$, between vertices $u$ and $v$ of $G$ is defined as the length of a shortest path between $u$ and $v$ in $G$. The distance matrix of $G$ is the matrix $\mathbf{D}(G) =[d_G(u, v)]_{u,v\in V}$. The second largest distance eigenvalue $\lambda_2(G)$ of $G$ is the second largest one in the spectrum of $\mathbf{D}(G)$.
  In this work, we completely characterize the connected graphs $G$  for which  $\lambda_2(G)<-1/2$ through approaches both spectral and structural.
\end{abstract}

\maketitle
\section{Introduction}\label{sec:intro}

The problem of characterizing connected graphs $G$ whose distance matrix has
second largest eigenvalue less than $-1/2$ has recently attracted considerable
attention in Spectral Graph Theory. Several papers have been published on this
topic and throughout this article, we highlight some of these contributions.
For a general overview of the origins and motivations of the problem, as well as
the current state of the art, we refer the reader to the article by Guo and Zhou
\cite{guo2024graphs} and to the preprint by Yang and Wang
\cite{yang2025tricyclicgraphssecondlargest}. Additional background on the
development of research related to the distance matrix of graphs can be found in
the surveys by Aouchiche and Hansen \cite{AOUCHICHE2014301} and by Lin
\textit{et al.} \cite{SurveyLinShu}.

In this work, we take as our starting point a result from \cite{guo2024graphs},
which establishes that any connected graph satisfying $\lambda_2 < -1/2$ must be
chordal. Based on this result and on structural aspects of chordal graphs, in
particular those concerning the cardinality and multiplicity of their minimal
vertex separators (see definitions in the text), we obtain a complete
characterization of connected graphs satisfying this spectral condition.
More specifically, since block graphs with $\lambda_2 < -1/2$ were already
characterized by Xue \textit{et al.} in \cite{Xue-Lin-Shu-block}, our focus here is
on connected chordal graphs that are not block graphs and that satisfy this
inequality.

The paper is organized as follows. In Section~2, we recall preliminary notions on
chordal graphs and on subclasses that play a central role in our analysis,
including Ptolemaic graphs, block graphs, and split graphs. We also summarize the
properties of the distance matrix that will be used in the subsequent sections.
Section~3 reviews known results concerning graphs whose distance matrix satisfies
$\lambda_2 < -1/2$. In Section~4, we provide a structural description of connected
chordal graphs satisfying this condition by proving the following results (see
definitions in the text):

\begin{quote}
\textit{Let $G$ be a connected chordal graph that is not a block graph and whose distance
matrix has the second largest eigenvalue less than $-1/2$. If $G$ has the diameter~2
then $G$ is a connected induced subgraph of a relaxed block star graph
(Theorem~\ref{theo:pto-mvs2}). If the diameter of $G$ is equal to 3, then $G$ is a connected
induced subgraph of either a $Pt_1$ graph or a $Pt_2(p,q)$ graph, with
$p,q \ge 2$ (Theorem~\ref{theo:pto-mvs3}). Moreover, any such graph $G$ has
diameter at most~3.}
\end{quote}

\noindent Finally, in Section~5 we address the spectral aspects of our characterization and
prove that relaxed block star graphs, as well as the graphs $Pt_1$ and $Pt_2(p,q)$,
indeed satisfy the condition $\lambda_2 < -1/2$. To this end, we apply Descartes’
Rule of Signs and Sturm’s Theorem, thereby completing the characterization of
connected chordal graphs whose distance matrix has second largest eigenvalue less
than $-1/2$, see Corollaries \ref{cor:caract-diam2} and \ref{cor:caract-diam3}.

\section{Preliminaries}
  Let $G=(V,E)$ be a graph with vertex set $V$ and edge set $E$.
The {\em set of neighbors\/} of  $v \in V$ is denoted by
$N_G(v) = \{ w \in V(G) \mid \{v,w\} \in E\}$. The vertex $v$ is said to be \textit{universal} in $G$ when  $N_G(v)\cup\{v\} =V$. 
For any $S \subset V$, let $G[S]$ be the subgraph of $G$
induced by $S$. If $G[S]$ is a complete graph then $S$ is a \textit{clique}.
For an integer $p\geq 1$, $pG$ denotes the disjoint union of $p$
copies of $G$. Also, the \textit{join} $G_1\oplus G_2$ of graphs $G_1$ and $G_2$ where $G_1=(V_1,E_1)$ and $G_2=(V_2,E_2)$ have disjoint vertex sets, is the graph union $G_1 \cup G_2$ together with all the edges joining
 the elements of $V_1$ and $V_2$.
A graph $G$ is said to be \textit{free of  $H$} or $H$-free if $G$ does not contain the graph $H$  as an induced subgraph.  We say then that the graph $H$ is a \textit{forbidden subgraph} of $G$.
\smallskip

From now on, we assume that all the graphs are simple and connected.

\subsection*{Chordal graphs}
A graph is \textit{chordal} if every cycle of length at least four has a chord,
where a \textit{chord} is an edge joining two non-adjacent vertices of the cycle.

Given a chordal graph $G=(V,E)$,  a vertex $v$ is said to be {\em simplicial\/} if $N(v)$  is a
clique in $G$.
A subset $S \subset V$  is a {\em vertex separator}  for
non-adjacent vertices $u$  and $v$  (a $uv$-{\em separator}) if the
removal of $S$ from the graph separates $u$ and $v$  into distinct
connected components. If no proper subset of $S$  is a
$uv$-separator then $S$ is a {\em minimal $uv$-separator}. When the
pair of vertices remains unspecified, we refer to $S$  as a {\em
minimal vertex separator}, a  $mvs$ for short.
The set of minimal vertex separators  of $G$ is denoted by $\mathbb{S}$.

\begin{lemma}\cite{Di61}\label{chordal-1}
    $G$ is chordal if and only if every minimal vertex separator of $G$ is a clique.
\end{lemma}

A {\em clique-tree} of a chordal graph $G$ is defined as a tree $T$  whose vertices
are the maximal cliques of $G$ and such that for every  two maximal cliques $Q$ and $Q^\prime$,
each clique in the path from $Q$ to $Q^\prime$ in $T$ contains $Q\cap Q^\prime$.
Blair and Peyton \cite{BP93}   proved that,
for a clique-tree $T=(V_T, E_T)$ of $G$,
a set $S\subset V$ is a minimal vertex separator of $G$ if
and only if $S= Q' \cap Q''$ for some edge $\{Q', Q''\}\in E_T$.
Moreover, the multiset  ${\mathbb M}$ of
the minimal vertex separators of $G$ is the same for every
clique-tree of $G$.
We define the {\em multiplicity} $\mu(S)$ of the minimal vertex separator $S$ as the number of times that $S$ appears in ${\mathbb
M}$.
The algorithm presented in Markenzon and Pereira \cite{MP10}
computes
the set $\mathbb S$ of minimal vertex separators  of a chordal graph  and
their multiplicities  in linear time.

Another concepts about chordal graphs 
can be found in Blair and Peyton \cite{BP93} and Golumbic \cite{Go04}.

\subsection*{Subclasses of chordal graphs}

We expose here the relations between chordal graphs and other well known classes of graphs, and refer to  chapter 6 of the book \cite{Beineke_Golumbic_Wilson_2021} for more details.

The \textit{distance} between vertices $u$ and $v$ in a connected graph  $G=(V,E)$, denoted by $d_G(u,v)$, is defined as the length of a shortest path linking $u$ and $v$ in $G$. The \textit{diameter of $G$}
 is the number $\mathrm{diam}(G) = \max\{d_G(v, u )\ | \ u,v \in V\}$.
The graph $G$ is \textit{distance hereditary} if the distance between any two vertices remains
 the same in every connected induced subgraph.
 The graph $G$ is \textit{Ptolemaic} if for any four vertices $u, v, w, x$ of $G$, it holds that
$d_G(u, v)d_G(w, x) \leq  d_G(u,w)d(v, x) + d_G(u, x)d_G(v,w)$.
Ptolemaic graphs are chordal graphs.
\smallskip

The \textit{gem} and the  \textit{diamond} are the graphs
illustrated in Figure \ref{fig-forbidden1}.

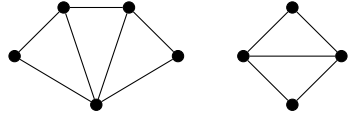
\begin{figure}[h]
\begin{center}
\begin{tikzpicture}
 [scale=.43,auto=left]
 \tikzstyle{every node}=[circle, draw, fill=black,
                        inner sep=1.5pt, minimum width=4pt]

  \node (c) at (1,10) {};
  \node (d) at (3,10) {};
  \node (b) at (-0.5,8.5)  {};
  \node (e) at (4.5,8.5)  {};
  \node (a) at (2,7) {};

  \foreach \from/\to in {c/d,c/b,c/a,d/a,d/e,b/a,e/a}
    \draw (\from) -- (\to);

     \node (i) at (8,10) {};
  \node (g) at (6.5,8.5) {};
 \node (h) at (9.5,8.5)  {};
  \node (f) at (8,7)  {};
  \foreach \from/\to in {i/g,i/h, g/h,g/f, h/f}
    \draw (\from) -- (\to);

   \end{tikzpicture}
\end{center}

\caption{Gem and  diamond graphs.}
\label{fig-forbidden1}
\end{figure}

\begin{lemma} \cite{Ho81}\label{lem:pto}
The following conditions are equivalent:
\begin{enumerate}
 \item[(i)] $G$ is a Ptolemaic graph.
\item[(ii)] $G$ is gem-free and chordal.
\item[(iii)] $G$ is distance hereditary and chordal.
\end{enumerate}
\end{lemma}

  A \textit{block } of a given graph is a maximal connected subgraph that has no cut vertex (a $mvs$ of cardinality one is also called a cut vertex). A  graph  is called a
\textit{block graph}  if all of its blocks are cliques. They are also the Ptolemaic graphs  in which every two vertices at distance  two from each other are connected by a unique shortest path.

\begin{lemma}\label{lem:block}
The following assertions are equivalent:
 \begin{enumerate}
     \item[(i)] $G$ is a block graph;
     \item[(ii)] $G$ is a chordal graph that is diamond-free;
     \item[(iii)] $G$ is a chordal graph and any minimal vertex separator is an unitary set.
 \end{enumerate}
\end{lemma}

The class of \textit{complement reducible graphs}, known as \textit{cographs}, is the
hereditary class of $P_4$-free graphs.

Since by definition, chordal graphs are $C_4$-free graphs,
the class of chordal cographs is the hereditary class of $\{C_4,P_4\}$-free graphs
whose elements are often called \textit{\textit{quasi}-threshold graphs} (\cite{YanChen1996})
and also constitute a subclass of Ptolemaic graphs.

\subsection*{Distance matrix and its spectral properties}
We next present some properties of the distance matrix of a graph that will be
used in the structural part of our characterization of graphs satisfying
$\lambda_2 < -1/2$.

Let $V=\{v_1,\ldots,v_n\}$. The \textit{distance matrix} of the graph
$G=(V,E)$ is the $n\times n$ matrix $\mathbf{D}(G)=[d_G(v_i,v_j)]$.
The eigenvalues of $\mathbf{D}(G)$ are called the \textit{distance eigenvalues}
of $G$. Since $\mathbf{D}(G)$ is symmetric, all its eigenvalues are real and may
be ordered as \
$
\lambda_1(G) \ge \lambda_2(G) \ge \cdots \ge \lambda_n(G).
$ \
Thus, $\lambda_2(G)$ denotes the second largest distance eigenvalue of $G$.

 \begin{lemma}[Cauchy's Interlace Theorem]\label{lem:interl} Let $\mathbf{A}$ be a symmetric matrix of order $n$,
and let $\mathbf{B}$ be a principal submatrix of $A$ of  order $m$. If $\lambda_1(A) \geq  \lambda_2(A) \geq  \ldots \geq  \lambda_n(A)$
lists the eigenvalues of $\mathbf{A}$ and $\lambda_1(B) \geq  \lambda_2(B) \geq \ldots \geq  \lambda_m(B)$ are the eigenvalues of $\mathbf{B}$
then
\[\lambda_{i} (A) \geq \lambda_i (B) \geq  \lambda_{n-m+i} (A), \] for $i = 1 \ldots m$.

 \end{lemma}

\begin{remark}\label{rem:interl}
Since Ptolemaic graphs are distance hereditary, the distance matrix of every connected induced subgraph $H$ of a Ptolemaic graph $G$ is a principal submatrix of $\mathbf{D}(G)$. This also  holds if $H$ is an  induced subgraph  with  $\mathrm{diam}(H) \leq 2$ of a general   graph, not necessarily a distance hereditary graph.
  In both cases,  by Lemma \ref{lem:interl}, we have \[\lambda_1(G) \geq \lambda_1(H) \geq \lambda_2(G) \geq \lambda_2(H).\]
Parti\-cularly, if $G$ is a
graph with $\lambda_2(G) < -1/2$,  $G$ is free of subgraphs $H$  such that $\lambda_2(H) >  -1/2$.

\end{remark}

 \section{Known results on graphs for which $\lambda_2 <-1/2$}

Some known results were very important for the development  of our work and are here established as lemmas.
The first of these results, due to Guo and Zhou \cite{guo2024graphs}, actually  inspires our research.

 \begin{lemma}[\cite{guo2024graphs}, Theorem 3.1]\label{lem:guochordal}
    If a  graph $G$ satisfies $\lambda_2(G) <-1/2$ then $G$ is a chordal graph.
\end{lemma}

\begin{figure}[h]
\begin{center}
\begin{tikzpicture} [scale=.25,auto=left]
 \tikzstyle{node}=[circle, draw, fill=black,
                        inner sep=1pt, minimum width=3pt]

\node[node] (a) at (0,24) {};
\node[node] (b) at (2,24) {};
\node[node] (c) at (4,24) {};
\node[node] (d) at (6,24) {};
\node[node] (e) at (1,26) {};
\node[node] (f) at (3,26) {};

\node at (3,22) [draw=none,fill=none] {$F_1$};

\node[node] (g) at (9,24) {};
\node[node] (h) at (11,24) {};
\node[node] (i) at (13,24) {};
\node[node] (j) at (15,24) {};
\node[node] (k) at (11,26) {};
\node[node] (l) at (13,26) {};

\node at (12,22) [draw=none,fill=none] {$F_2$};

\node[node] (m) at (18,24) {};
\node[node] (n) at (20,24) {};
\node[node] (o) at (22,24) {};
\node[node] (p) at (24,24) {};
\node[node] (q) at (26,24) {};
\node[node] (r) at (22,26) {};

\node at (24,22) [draw=none,fill=none] {$F_3$};

\node[node] (s) at (29,24) {};
\node[node] (t) at (31,24) {};
\node[node] (u) at (33,24) {};
\node[node] (v) at (35,24) {};
\node[node] (x) at (37,24) {};
\node[node] (y) at (39,24) {};
\node[node] (z) at (31,26) {};

\node at (34,22) [draw=none,fill=none] {$F_4$};

\node[node] (a1) at (0,16) {};
\node[node] (b1) at (0,20) {};
\node[node] (c1) at (2,18) {};
\node[node] (d1) at (4,18) {};
\node[node] (e1) at (6,20) {};
\node[node] (f1) at (6,16) {};

\node at (3,14) [draw=none,fill=none] {$F_5$};

\node[node] (g1) at (9,16) {};
\node[node] (h1) at (9,20) {};
\node[node] (i1) at (11,16) {};
\node[node] (j1) at (11,20) {};
\node[node] (k1) at (15,16) {};
\node[node] (l1) at (15,20) {};
\node[node] (m1) at (13,18) {};

\node at (12,14) [draw=none,fill=none] {$F_6$};

\node[node] (n1) at (18,14) {};
\node[node] (o1) at (18,20) {};
\node[node] (p1) at (22,14) {};
\node[node] (q1) at (22,20) {};
\node[node] (r1) at (20,17) {};
\node[node] (s1) at (23,17) {};
\node[node] (t1) at (26,17) {};

\node at (24,14) [draw=none,fill=none] {$F_7$};

\node[node] (u1) at (30,16) {};
\node[node] (v1) at (30,20) {};
\node[node] (x1) at (32,18) {};
\node[node] (y1) at (34,16) {};
\node[node] (w1) at (34,20) {};
\node[node] (z1) at (37,19) {};
\node[node] (zz1) at (37,21) {};

\node at (34,14) [draw=none,fill=none] {$F_8$};

\node[node] (a2) at (0,8) {};
\node[node] (b2) at (2,6) {};
\node[node] (c2) at (2,10) {};
\node[node] (d2) at (4,8) {};
\node[node] (e2) at (6,8) {};

\node at (3,4) [draw=none,fill=none] {$F_9$};

\node[node] (f2) at (9,8) {};
\node[node] (g2) at (11,6) {};
\node[node] (h2) at (11,10) {};
\node[node] (i2) at (13,8) {};
\node[node] (j2) at (15,8) {};
\node[node] (k2) at (17,8) {};

\node at (12,4) [draw=none,fill=none] {$F_{10}$};

\node[node] (l2) at (22,6) {};
\node[node] (m2) at (26,6) {};
\node[node] (n2) at (20,10) {};
\node[node] (o2) at (24,10) {};
\node[node] (p2) at (28,10) {};

\node at (24,4) [draw=none,fill=none] {$F_{11}$};

\node[node] (q2) at (31,8) {};
\node[node] (r2) at (33,6) {};
\node[node] (s2) at (33,10) {};
\node[node] (t2) at (36,6) {};
\node[node] (u2) at (36,10) {};

\node at (34,4) [draw=none,fill=none] {$F_{12}$};

\node[node] (a3) at (0,0) {};
\node[node] (b3) at (2,-2) {};
\node[node] (c3) at (2,2) {};
\node[node] (d3) at (5,-2) {};
\node[node] (e3) at (5,2) {};
\node[node] (f3) at (7,0) {};

\node at (3,-4) [draw=none,fill=none] {$F_{13}$};

\foreach \from/\to in {a/b,b/c,c/d,b/e,b/f,
g/h,h/i,i/j,h/k,i/l,
m/n,n/o,o/p,p/q,o/r,
s/t,t/u,u/v,v/x,x/y,t/z,
a1/b1,a1/c1,b1/c1,c1/d1,d1/e1,d1/f1,g1/i1,h1/j1,i1/j1,i1/m1,j1/m1,m1/k1,m1/l1,
n1/r1,n1/p1,o1/r1,o1/q1,p1/r1,q1/r1,r1/s1,s1/t1,
u1/x1,u1/v1,v1/x1,x1/y1,x1/w1,y1/w1,w1/z1,w1/zz1,
a2/b2,a2/c2,b2/c2,b2/d2,c2/d2,d2/e2,
f2/h2,f2/g2,f2/i2,g2/i2,h2/i2,i2/j2,j2/k2,
l2/m2,l2/n2,l2/o2,m2/o2,m2/p2,n2/o2,o2/p2,
q2/r2,q2/s2,r2/s2,r2/t2,r2/u2,s2/t2,s2/u2,
a3/b3,a3/c3,b3/c3,b3/d3,b3/e3,c3/d3,c3/e3,d3/e3,d3/f3,e3/f3}
\draw (\from) -- (\to);

 \end{tikzpicture}
\vskip -.5cm
\caption{Forbidden subgraphs for graphs satisfying $\lambda_2 <-1/2$ \cite{guo2024graphs}.}
\label{fig:grande}
\end{center}
\end{figure}
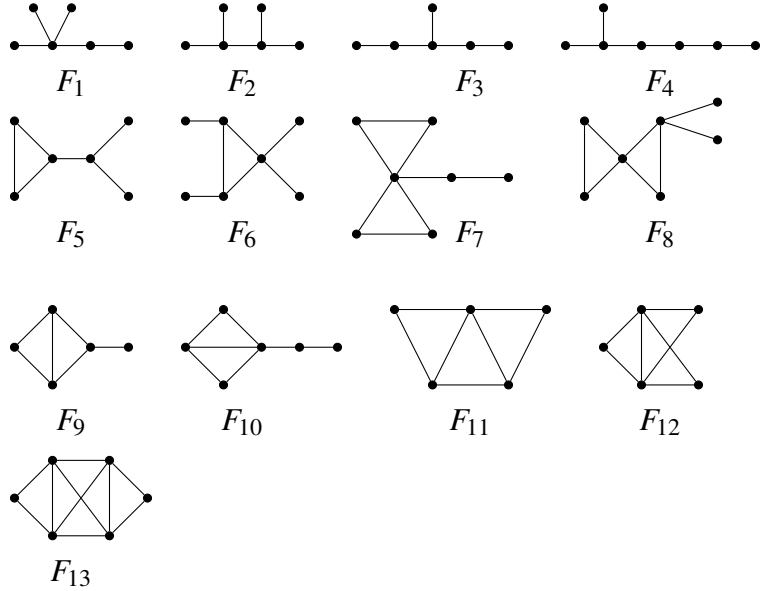

Several graphs, depicted in Figure \ref{fig:grande},
were  identified as  forbidden subgraphs of graphs satisfying $\lambda_2 <-1/2$ in  \cite{guo2024graphs}.
In  the same paper,  the split graphs satisfying this condition     were characterized, see Theorem \ref{theo:split}.
\smallskip

A \textit{block star} is a block graph whose all blocks contain a
common vertex.
A block graph $G$ is \textit{loose} if $\mu(S)=1$ for every $mvs$ $S \in {\mathbb S}$.
Let $BG(p, q, 3, 2, 2)$ with $p, q \geq 2$ and $BGA$ be the two block graphs  shown in Figure \ref{fig:block}. Another result, also important for our work,  characterize the  block graphs satisfying $\lambda_2 <-1/2$.

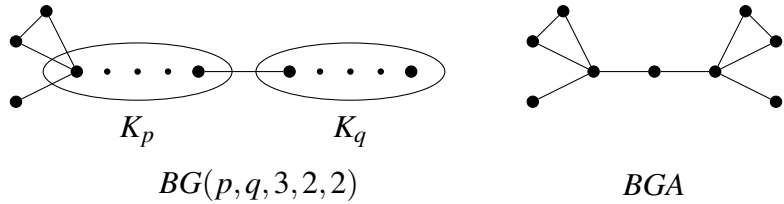
\begin{figure}[h]
\begin{center}
\begin{tikzpicture} [scale=.20,auto=left]
\tikzstyle{block} = [ellipse, draw,
text width=3.5em, text centered,  minimum height=1.8em]
 \tikzstyle{node}=[circle, draw, fill=black,
                        inner sep=1.5pt, minimum width=4pt]
 \tikzstyle{nodesm}=[circle, draw, fill=black,
                        inner sep=.7pt, minimum width=1pt]

\node[node] (a) at (0,4) {};
\node[node] (b) at (0,8) {};
\node[node] (c) at (2,10) {};
\node[node] (d) at (4,6) {};
\node[nodesm] (e) at (6,6) {};
\node[nodesm] (f) at (8,6) {};
\node[nodesm] (g) at (10,6) {};
\node[node] (h) at (12,6) {};

\node[block] (K1)    at (8,6) {};
\node at (8,2) [draw=none,fill=none] {$K_p$};

\node[node] (i) at (18,6) {};
\node[nodesm] (j) at (20,6) {};
\node[nodesm] (k) at (22,6) {};
\node[nodesm] (l) at (24,6) {};
\node[node] (m) at (26,6) {};

\node at (22,2) [draw=none,fill=none] {$K_q$};
\node[block] (K2)    at (22,6) {};

\node at (16,-1.5) [draw=none,fill=none] {$BG(p,q,3,2,2)$};

\node[node] (n) at (34,4) {};
\node[node] (o) at (34,8) {};
\node[node] (p) at (36,10) {};
\node[node] (q) at (38,6) {};
\node[node] (r) at (42,6) {};
\node[node] (s) at (46,6) {};
\node[node] (t) at (50,4) {};
\node[node] (u) at (50,8) {};
\node[node] (v) at (48,10) {};

\node at (42,-1.5) [draw=none,fill=none] {$BGA$};

\foreach \from/\to in {a/d,b/d,b/c,c/d,h/i,
n/q,o/q,o/p,p/q, q/r,r/s,s/u,s/v,s/t,u/v}
\draw (\from) -- (\to);

 \end{tikzpicture}
\caption{Graphs $BG(p,q,3,2,2)$ and $BGA$ of Lemma \ref{lem:block-graph}.}
\label{fig:block}
\end{center}
\end{figure}


\begin{lemma}\cite{Xue-Lin-Shu-block}\label{lem:block-graph}
  Let $G$ be a block graph. Then $\lambda_2(G) <-1/2$
 if and only if
   $G$ is a block star, or
  $G$ is a loose block graph, or
$G$ is a non trivial connected induced subgraph of $BG(p, q, 3, 2, 2)$, or
$G$ is a non trivial connected induced subgraph of $BGA$.
 \end{lemma}

For sake of completeness, we may register that \cite{Xing-Zhou16}   characterized  all graphs for which $\lambda_2 < 2-\sqrt{2} \approx 	-0.58578$ as well as all trees and all  unicyclic graphs satisfying  $\lambda_2<-1/2$ holds. Also, in \cite{Liu-Xue-Guo}, the graphs satisfying
$\lambda_2 \leq (17- \sqrt{329})/2 \approx -0.5692$ were proved to be
 determined by their $\mathbf{D}$-spectra.

\section{Structural description of graphs satisfying $\lambda_2 <-1/2$}\label{sec:4-struct}

 Theorem 5.1 of Guo and Zhou \cite{guo2024graphs} characterizes the split graphs whose second largest distance
eigenvalue satisfies $\lambda_2 <-1/2$.
The following theorem, stated in terms of minimal vertex separators, provides an alternative
formulation of Theorem 5.1 of  \cite{guo2024graphs} and  is  the version adopted in the present paper.

We may recall that a \textit{split graph} is a connected  graph whose vertices can be partitioned into a clique and an independent vertex set. Split graphs are chordal graphs.
Consider  $SP^t$ and $ SP_1$  the graphs  depicted in Figure \ref{fig:split}.

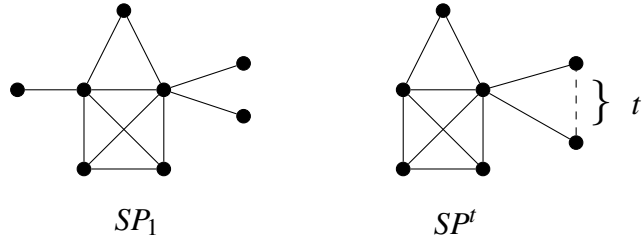
\begin{figure}[h]
\begin{center}
\begin{tikzpicture} [scale=.35,auto=left]
 \tikzstyle{every node}=[circle, draw, fill=black,
                        inner sep=1.8pt, minimum width=4pt]

 \node (o) at (17,2) {};
 \node (p) at (14,5)  {};
 \node (q) at (14,2)  {};
 \node (r) at (17,5) {};
 \node (t) at (15.5,8) {};
  \node (s) at (11.5,5) {};
\node (y) at (20,4) {};
\node (z) at (20,6) {};

\node at (16,0) [draw=none,fill=none] {$SP_1$};

\foreach \from/\to in {o/p,o/r,o/q,p/r,p/q,q/r,s/p,r/y,r/z,t/p,t/r}
\draw (\from) -- (\to);

 \node (o1) at (29,2) {};
\node (p1) at (26,5)  {};
 \node (q1) at (26,2)  {};
\node (r1) at (29,5) {};
 \node (t1) at (27.5,8) {};
\node (y1) at (32.5,3) {};
\node (z1) at (32.5,6) {};

\foreach \from/\to in {o1/r1,o1/q1,o1/p1,p1/r1,p1/q1,q1/r1,r1/y1,r1/z1,t1/p1,t1/r1}
\draw (\from) -- (\to);

 \draw[dashed] (y1) -- (z1) ;
\node at (34,4.5) [draw=none,fill=none] {{\Large \} } $t$};

\node at (28,0) [draw=none,fill=none] {$SP^t$};
     \end{tikzpicture}
\caption{Split graphs of Theorem \ref{theo:split}.}
 \label{fig:split}
\end{center}
\end{figure}

\begin{theorem}\label{theo:split}
 Let $G$ be a split graph and $\mathbb S$ its set of mvs.
Then  $\lambda_2(G) <  -1/2$ if and only if
 $G \cong  SP_1$ \ or \ $G \cong SP^t$, $t \geq 0$, as depicted in Figure \ref{fig:split}, \ or \
    $G$ is also a block graph containing a block $K_s$, $s\geq 2$, such that
\begin{enumerate}
\item[(i)] $|{\mathbb S} | = 1$ or
\item[(ii)] $|{\mathbb S} | = 2$ with $\mu(S_1) =1$ and $ \mu(S_2) \leq 2$ or
\item[(iii)] $\mu(S_i) =1$,  $ S_i \in {\mathbb S}$ or
\end{enumerate}
  any split induced subgraph of any of these graphs.
\end{theorem}

\begin{proof}
    It is straightforward proving that the families of graphs described in Theorem~5.1 of
Guo and Zhou~\cite{guo2024graphs} coincide exactly with those listed in
Theorem~\ref{theo:split}.
\end{proof}
\smallskip

As in the previous theorem, the development of our results is based on a detailed investigation of the minimal vertex se\-pa\-ra\-tors of the graphs.
If a graph  satisfies $\lambda_2 <-1/2$ and contains only $mvs$ of cardinality one it is a block graph, by Lemma \ref{lem:block}, and the problem is already solved, see  Lemma \ref{lem:block-graph}. We may note that, in this case, the graph can   have any diameter.
Our goal from now is to determine which chordal graphs that are not block graphs have $\lambda_2 < - 1/2$. In the following results,  we   study the graphs considering their diameters and characteristics of their \textit{mvs}.

\begin{proposition}\label{lem:mvs-cardinal3}

If a chordal graph $G$ has a \textit{mvs} of cardinality 3 then  it has $\lambda_2(G) > -1/2$.
\end{proposition}

\begin{proof}
Let us consider $G$ to be the chordal graph  with  minimum number of vertices containing a $mvs$  $S=\{a,b,c\}$ of cardinality 3,  depicted in Figure \ref{fig:mvs3}(a).
This graph has $\lambda _2\approx -0.44949$.
As all chordal graphs containing at least one $mvs$ of cardinality 3 have $G$ as an induced subgraph, according to Remark \ref{rem:interl},
these graphs do not satisfy the condition $\lambda_2 <-1/2$.
\end{proof}

It is interesting to observe that the graph in Figure \ref{fig:mvs3}(a) is a new forbidden subgraph for our problem.

\begin{figure}[h]
\begin{center}
\begin{tikzpicture}
 [scale=.48,auto=left]
 \tikzstyle{every node}=[circle, draw, fill=black,
                        inner sep=1.5pt, minimum width=4pt]

  \node (c) at (1,10) {};
  \node (d) at (3,10) {};
  \node (b) at (-0.5,8.5)  {};
  \node (e) at (4.5,8.5)  {};
  \node (a) at (2,7) {};

\node at (2,6) [draw=none,fill=none] {$a$};
\node at (1,11) [draw=none,fill=none] {$b$};
\node at (3,11) [draw=none,fill=none] {$c$};

\node at (2,4.5) [draw=none,fill=none] {$(a)$};

 \foreach \from/\to in {c/d,c/b,c/a,d/a,d/e,b/a,b/d,e/a,e/c}
 \draw (\from) -- (\to);

  \node (i) at (9,10) {};
 \node (g) at (7.5,8.5) {};
\node (h) at (10.5,8.5)  {};
 \node (f) at (9,7)  {};
 \node (j) at (9,11.5)  {};

\node at (9,4.5) [draw=none,fill=none] {$(b)$};

 \foreach \from/\to in {i/g,i/h, g/h,g/f, h/f,g/j,h/j}
  \draw (\from) -- (\to);

\end{tikzpicture}
\end{center}
\caption{(a) Smallest  graph with a $mvs$ $S$ such that $|S|=3$; (b) smallest graph with a \textit{mvs} $S$ for which $|S|=2$ and $\mu(S)=2$.}
\label{fig:mvs3}
\end{figure}
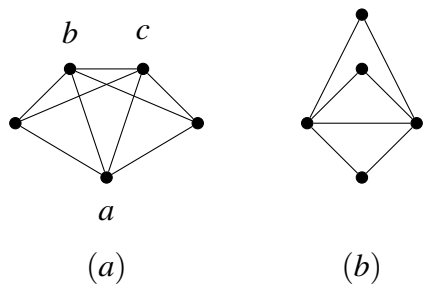

\begin{proposition}\label{lem:mvs-cardinal}
Let $G$ be a chordal graph.
If $G$ has a $mvs$ $S$ of  cardinality 2 with multiplicity
at least 2   then $\lambda_2(G) > -1/2$.
\end{proposition}
\begin{proof}
The chordal graph with minimum number of vertices which  has a $mvs$  of  cardinality 2 with multiplicity
at least 2
 is in Figure \ref{fig:mvs3}(b). It corresponds to $F_{12}$ in  Figure \ref{fig:grande}, so it does not obey the condition $\lambda_2 < -1/2$.
Any chordal graph with more vertices and a $mvs$ of cardinality 2 with multiplicity
at least 2 has this graph as an induced subgraph.
Thus, graphs with  $mvs$ of cardinality 2  with multiplicity greater than one
 does not satisfy the condition $\lambda_2 < -1/2$.
\end{proof}

\begin{proposition}\label{lem:Ptolemaic}
If the graph $G$ satisfies $\lambda_2(G) < -1/2$ then $G$ is a Ptolemaic graph.
\end{proposition}
\begin{proof}
It is already known that if $\lambda_2(G) < -1/2$ then $G$ is a chordal graph (Lemma \ref{lem:guochordal}).
Also, the forbidden subgraph $F_{11}$ (Figure \ref{fig:grande}) can not occur in Ptolemaic graphs, as seen in Lemma \ref{lem:pto}.
So, the graphs that obey the condition are Ptolemaic graphs.
\end{proof}

\section*{Ptolemaic graphs of diameter 2 satisfying $\lambda_2 <-1/2$}

The block graphs of diameter 2  for which $\lambda_2 <-1/2$ are the so called ``block stars'' of Lemma \ref{lem:block-graph}. In order to describe all the graphs of diameter 2  satisfying the above condition, we  define a new class of graphs, the {\em relaxed block stars}.
We call  {\em full house}  the graph consisting of the \emph{house graph},
well known in the literature, plus  two chords
connecting diagonally opposite vertices of the cycle of size 4.
A {\em relaxed block star}   is a graph with an universal vertex
and such that it has at least one $mvs$ with cardinality 2 and all its blocks are cliques
or full houses or diamonds. See  Figure \ref{fig:examplerelaxed} for an example.

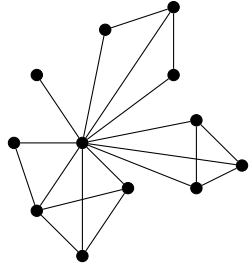
\begin{figure}[h]
\begin{center}
\begin{tikzpicture} [scale=0.15,auto=left]
 \tikzstyle{every node}=[circle, draw, fill=black,
                        inner sep=1.5pt, minimum width=4pt]

\node (a) at (6,10) {};
 \node (b) at (0,10)  {};
\node (c) at (2,4)  {};
\node (d) at (6,0) {};
 \node (e) at (10,6) {};

\node (f) at (2,16) {};

\node (g) at (8,20) {};
\node (h) at (14,22) {};
\node (i) at (14,16) {};

\node (j) at (16,12) {};
\node (k) at (16,6) {};
\node (l) at (20,8) {};


\foreach \from/\to in {a/b,a/c,a/d,a/e,b/c,c/d,c/e,d/e,
a/f,
a/g,a/h,a/i,g/h,h/i,
a/j,a/k,a/l,j/k,j/l,k/l}
\draw (\from) -- (\to);

 \end{tikzpicture}
 \caption{Example of a relaxed block star graph.}
 \label{fig:examplerelaxed}
\end{center}
\end{figure}

We prove in Section~\ref{sec:spec-results} (Theorem~\ref{the:main}) that every
relaxed block star graph satisfies the condition $\lambda_2 < -1/2$.
The next theorem shows that the  graphs of diameter~2 satisfying this condition
and that are not block graphs are precisely the connected induced subgraphs of
relaxed block star graphs.

\smallskip

\begin{theorem} \label{theo:pto-mvs2}
 Let $G$ be a  graph of diameter 2 and suppose that
its set \ $\mathbb S$ of minimal vertex separators  has at least one $mvs$ of cardinality 2.
If $\lambda_2(G) < -1/2$ then $G$ is a connected  induced subgraph of a relaxed block star.
\end{theorem}

\begin{proof}
Let $G$ be  a   graph with $\mathrm{diam}(G)=2$ and at least one $mvs$ of cardinality 2.
By Lemma \ref{lem:guochordal},    we know that
the condition $\lambda_2(G) <-1/2$ implies that $G$ is a chordal graph and thus, by Proposition
\ref{lem:mvs-cardinal3}, the $mvs$ of $G$ have cardinality 1 or 2. Also, by Proposition \ref{lem:pto}, $G$ is a Ptolemaic graph. Thus, in the sequence, we analyze all the possibilities among Ptolemaic graphs of diameter 2  with  $mvs$ of cardinalities 1 and   2  in order to conclude that we must have a connected induced subgraph of a relaxed block star.

Firstly,  let   $G$ be a  graph   with  at least two $mvs$ of cardinality 1.
 Suppose that the two $mvs$ belong  to  the same maximal clique $Q$ of $G$.
Then  there is at least one vertex, that do not belong to $Q$,
adjacent to each $mvs$.
The existence of these vertices implies a diameter  at least  3.
If the two $mvs$s do not belong to the same maximal clique
there must exist another $mvs$ in their path and
this also imply  a diameter at least 3.
So $G$ can have only one $mvs$ of cardinality 1.
Suppose now that $G$ has one $mvs$ $S$ of cardinality 2. By Proposition \ref{lem:mvs-cardinal}, it suffices to analyze   the case where the  $mvs$ $S$ has multiplicity one.
The Ptolemaic graphs with one $mvs$ of cardinality 2 and minimum number of vertices (4, 5 and 6 vertices)
appear in Figure \ref{fig:mvs-2}(a), (b), (c) and (d).
\begin{figure}[h]
\begin{center}
\begin{tikzpicture} [scale=0.3]
 \tikzstyle{every node}=[circle, draw, fill=black,
                        inner sep=1.8pt, minimum width=4pt]

  \node (x) at (-12,2) {};
  \node (y) at (-12,6) {};
\node (z) at (-10,4)  {};
\node (t) at (-14,4)  {};

\node at (-12,0) [draw=none,fill=none] {$(a)$};

\foreach \from/\to in {x/y,x/z,x/t,y/z,y/t}
    \draw (\from) -- (\to);

 \node (g1) at (-4,2) {};
  \node (h1) at (-4,5) {};
  \node (i1) at (-7,2)  {};
  \node (j1) at (-7,5)  {};
 \node (k1) at (-2,3.5) {};

\node at (-5,0) [draw=none,fill=none] {$(b)$};

\foreach \from/\to in {g1/h1,g1/i1,g1/j1,g1/k1,i1/j1,i1/h1,h1/k1,h1/j1}
    \draw (\from) -- (\to);

  \node (a) at (3,2) {};
  \node (b) at (6,2) {};
  \node (c) at (3,5)  {};
  \node (d) at (6,5)  {};
  \node (e) at (4.5,8) {};
  \node (f) at (1,3.5) {};

\node at (3,0) [draw=none,fill=none] {$(c)$};

\foreach \from/\to in {a/b,a/c,a/d,b/c,b/d,c/d,e/a,e/b,e/c,e/d,f/a,f/c}
    \draw (\from) -- (\to);

 \node (g) at (10,2) {};
  \node (h) at (10,5) {};
  \node (i) at (13,2)  {};
  \node (j) at (13,5)  {};
  \node (k) at (16,2) {};
\node (m) at (16,5) {};

\node at (13,0) [draw=none,fill=none] {$(d)$};

\foreach \from/\to in {g/h,g/i,g/j, h/i,h/j, i/j,k/i,k/j,k/m,m/i,m/j}
 \draw (\from) -- (\to);

 \end{tikzpicture}
\caption{Smallest Ptolemaic  graphs with one $mvs$ of cardinality 2.}
\label{fig:mvs-2}
\end{center}
\end{figure}
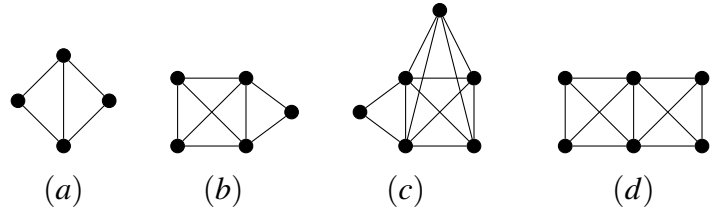

 Figure \ref{fig:mvs-2}(a) (the diamond graph) is an induced subgraph of the full house (Figure \ref{fig:mvs-2}(b)). Since this last one is an induced subgraph of the split graph $SP_1$, by Theorem \ref{theo:split} both the graphs have $\lambda_2 <-1/2$. Since they can be  induced subgraphs of any relaxed block star, they are allowed as subgraphs of our $G$.
The graph in Figure \ref{fig:mvs-2}(c)  is a split graph that does not satisfy Theorem \ref{theo:split}.  The graph in Figure \ref{fig:mvs-2}(d) has  $\lambda_2 \approx  -0.4641$, so it does not satisfy the required condition on $\lambda_2$.
 Figures \ref{fig:mvs-2}(c) and (d) correspond to the addition of one vertex to each of the maximal cliques
of the full house graph.
Any other addition, as, for instance, a vertex adjacent to a vertex of the $mvs$,
 results in the creation of a second $mvs$, a situation that we are going to study later.
Any Ptolemaic graph with one $mvs$ $S=\{a,b\}$ of cardinality $2$ and at least five vertices adjacent
simultaneously to $a$ and $b$  has the graphs of
Figure \ref{fig:mvs-2}(c) or Figure \ref{fig:mvs-2}(d) as an induced subgraph.
Hence, none of these graphs can  appear as a induced subgraph of our initial graph $G$.

So, being $G$ a  Ptolemaic graph of diameter 2,
if $\lambda_2(G) < -1/2$ and $G$ has only one $mvs$ and this $mvs$ has cardinality 2 then $G$
is the full house or an induced subgraph of the full house.

 Suppose now that the graph $G$ has more than one $mvs$.
  Firstly, consider that  there are  two $mvs$, $S_i$ and $S_j$, of cardinality 2.
If $S_i \cap S_j = \emptyset$ they can belong to same maximal clique (Figure \ref{fig:mvs2-2}(a))
or not.
In the first case, the graph is  the forbidden graph $F_{13}$  (Figure \ref{fig:grande}).
In the second case, there is at least  another $mvs$ separating the maximal cliques
and the graph has at least diameter 3 (Figure \ref{fig:mvs2-2}(b)). So, the situation $S_i \cap S_j = \emptyset$ does not happen.
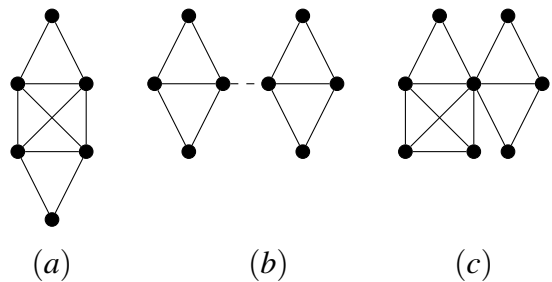
\begin{figure}[h]
\begin{center}
\begin{tikzpicture} [scale=.3,auto=left]
 \tikzstyle{every node}=[circle, draw, fill=black,
                        inner sep=1.8pt, minimum width=4pt]

\node (a1) at (-6,2) {};
 \node (b1) at (-9,5)  {};
 \node (c1) at (-9,2)  {};
\node (d1) at (-6,5) {};
 \node (e1) at (-7.5,8) {};
\node (f1) at (-7.5,-1) {};

\node at (-7.5,-3) [draw=none,fill=none] {$(a)$};

\foreach \from/\to in {a1/b1,a1/d1,a1/c1,b1/d1,b1/c1,c1/d1,e1/d1,b1/e1,f1/c1,f1/a1}
\draw (\from) -- (\to);

\node (a2) at (-1.5,2) {};
\node (b2) at (-3,5)  {};
\node (d2) at (0,5) {};
\node(d21) at (2,5) {};
\node (e2) at (-1.5,8) {};
\node (f2) at (5,5) {};
\node(g2) at (3.5,2){};
\node(h2) at (3.5,8){};

\foreach \from/\to in {a2/b2,a2/d2,b2/d2,e2/d2,b2/e2,f2/d21,h2/d21,g2/d21,h2/f2,g2/f2}
\draw (\from) -- (\to);

\draw[dashed] (d2) -- (d21) ;

\node at (2,-3) [draw=none,fill=none] {$(b)$};

\node (a3) at (11,2) {};
\node (b3) at (8,5)  {};
\node (c3) at (8,2)  {};
\node (d3) at (11,5) {};
\node (e3) at (9.5,8) {};
\node (f3) at (14,5) {};
\node(g3) at (12.5,2){};
\node(h3) at (12.5,8){};

\node at (11,-3) [draw=none,fill=none] {$(c)$};

\foreach \from/\to in {a3/b3,a3/d3,a3/c3,b3/d3,b3/c3,c3/d3,e3/d3,b3/e3,f3/d3,h3/d3,g3/d3,h3/f3,g3/f3}
\draw (\from) -- (\to);

 \end{tikzpicture}
\vskip -.5 cm
\caption{Smallest Ptolemaic graphs with two  $mvs$ of cardinality 2.}
\label{fig:mvs2-2}
\end{center}
\end{figure}

If $S_i \cap S_j = \{x\}$ and $\{x\}$  is not a $mvs$,
the graph  has the gem $F_{11}$ as induced subgraph, which is a forbidden  subgraph
(Figura \ref{fig:grande}).
If $\{x\}$  is a $mvs$ and $S_i$ and $S_j$ belong to full houses or
induced subgraphs of full houses as in Figure \ref{fig:mvs2-2}(c),  the graph  is a connected induced subgraph of a relaxed block star.

 Consider now that the graph has minimal vertex separators of cardinalities 1 and 2. Specifically,
suppose
that  $G$ has one $mvs$ of cardinality $1$, $S_i$, and at least one $mvs$ of cardinality $2$, $S_j=\{a,b\}$.
Figure \ref{fig:mvs-1and2} shows the possible cases.
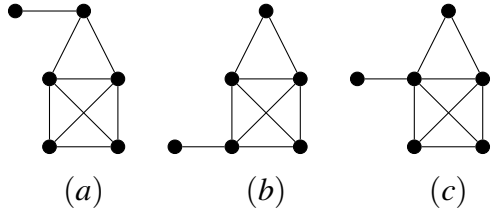
\begin{figure}[h]
\begin{center}
\begin{tikzpicture} [scale=.3,auto=left]
 \tikzstyle{every node}=[circle, draw, fill=black,
                        inner sep=1.8pt, minimum width=4pt]

 \node (a) at (9,2) {};
 \node (b) at (6,5)  {};
 \node (c) at (6,2)  {};
 \node (d) at (9,5) {};
 \node (e) at (7.5,8) {};
\node (f) at (3.5,2) {};

\node at (7.5,0) [draw=none,fill=none] {$(b)$};

\foreach \from/\to in {a/b,a/d,a/c,b/d,b/c,c/d,f/c,e/d,b/e}
\draw (\from) -- (\to);

\node (a1) at (1,2) {};
 \node (b1) at (-2,5)  {};
 \node (c1) at (-2,2)  {};
 \node (d1) at (1,5) {};
 \node (e1) at (-0.5,8) {};
\node (f1) at (-3.5,8) {};

\node at (-0.5,0) [draw=none,fill=none] {$(a)$};

\foreach \from/\to in {a1/b1,a1/d1,a1/c1,b1/d1,b1/c1,c1/d1,f1/e1,e1/d1,b1/e1}
\draw (\from) -- (\to);

\node (o) at (17,2) {};
 \node (p) at (14,5)  {};
 \node (q) at (14,2)  {};
 \node (r) at (17,5) {};
 \node (t) at (15.5,8) {};
  \node (s) at (11.5,5) {};

\node at (15.5,0) [draw=none,fill=none] {$(c)$};

\foreach \from/\to in {o/p,o/r,o/q,p/r,p/q,q/r,s/p,t/r,p/t}
\draw (\from) -- (\to);

 \end{tikzpicture}
\vskip -.5cm
\caption{Graphs with $mvs$ of cardinality 1 and 2.}
\label{fig:mvs-1and2}
\end{center}
\end{figure}

In Figures \ref{fig:mvs-1and2}(a) and \ref{fig:mvs-1and2}(b), $S_i \cap S_j = \emptyset$
and the  graphs contain the forbidden graph $F_9$ as an induced subgraph; so, they can not be induced subgraphs of graph $G$ due the condition $\lambda_2(G)<-1/2$.
We observe that, in this case,    the vertices adjacent simultaneously to $a$ and $b$ are  simplicial vertices.
For the graph in Figure \ref{fig:mvs-1and2}(c),
it holds that $\lambda_2 \approx -0.51210$ and so, it can appear as a connected induced subgraph of a relaxed block star.
 So, being $G$ a  Ptolemaic graph with $\mathrm{diam}(G)=2$ and at least one  $mvs$
of cardinality 2,
if $\lambda_2(G) < -1/2$ then $G$ has at most one $mvs$  of cardinality 1
that must belong to all $mvs$ of cardinality 2.
Hence, $G$ contains only one universal vertex.
The $mvs$ of cardinality 2 are such that
  they establish a full house graph or an induced subgraph of a full house graph
and all vertices  adjacent to  both vertices of the separators
are simplicial vertices.
We conclude that graph $G$ is a connected induced subgraph of a relaxed block graph.
\end{proof}

The next proposition establishes the equivalence between  the class of
Ptolemaic graphs of diameter 2 and that of the  \emph{quasi}-threshold graphs.
This equivalence  reinforces the fact  that any result obtained here for Ptolemaic graphs
of diameter 2 is valid for the more general class of cographs.

\begin{proposition} A graph $G$ is a Ptolemaic graph of diameter 2 if and only if $G$ is a connected \emph{quasi}-threshold graph.
\end{proposition}
\begin{proof}  If $G$ is a Ptolemaic graph of diameter 2 then it is a $P_4$-free graph, hence a \emph{quasi}-threshold graph.
Indeed, suppose by contradiction that $G$ has an induced $P_4$, $ a b c d$.
Then every pair of non-adjacent vertices must have a common neighbor.
Let $x$ be the neighbor of $a$ and $d$.
They form a cycle of size 5, and as $G$ is chordal, this cycle must have at least two chords.
Additional edges between the vertices $a,b,c,d$ are not allowed.
So, the sole possible edges are $\{x,b\}$ and $\{x,c\}$, forming the gem graph, which is a forbidden subgraph. The converse is straightforward.
 \end{proof}

\section*{Ptolemaic graphs with diameter 3 satisfying $\lambda_2 <-1/2$}

Before presenting the results of this subsection, we register here  the following general fact obtained from our previous studies for further reference.

\begin{lemma}\label{lem:full-diamond}
Let $G$ be a graph with $\lambda_2(G) < -1/2$. Then $G$ is a Ptolemaic graph and if $G$ has at least one $mvs$  of cardinality 2 then $G$ has  a full house graph or a diamond graph as an induced subgraph; furthermore,
the vertices adjacent to the vertices of the minimal vertex separators of cardinality 2 of $G$ are simplicial vertices.
\end{lemma}

\begin{proof}
Let $G$ be a Ptolemaic graph with at least one  $mvs$ of cardinality 2. By the proof of Theorem \ref{theo:pto-mvs2} and observing Figures \ref{fig:mvs-2} and \ref{fig:mvs-1and2}, $G$ contains a full house or a diamond as an induced subgraph, independently of the diameter of the graph. Also,
it is straightforward  that the vertices adjacent
to the $mvs$ of cardinality 2 must be simplicial vertices.
\end{proof}

In the sequence, in order to study the case of Ptolemaic graphs with diameter 3 satisfying  $\lambda_2 <-1/2$ that are not block graphs,  a new class of those graphs must be defined.
A {\em triple block graph} $Pt_2(p,q)$, where $p,q \geq 2$, is a graph with three blocks such that one of them
is a full house (or a diamond),  the other two are cliques $K_p$ and $K_q$ and
these cliques join the full house by the two vertices of its $mvs$,  see Figure \ref{fig:PT1-PT2}.

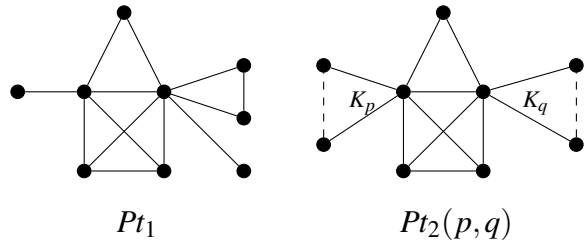
\begin{figure}[h]
\begin{center}
\begin{tikzpicture} [scale=.35,auto=left]
 \tikzstyle{every node}=[circle, draw, fill=black,
                        inner sep=1.8pt, minimum width=4pt]

 \node (o) at (17,2) {};
 \node (p) at (14,5)  {};
 \node (q) at (14,2)  {};
 \node (r) at (17,5) {};
 \node (t) at (15.5,8) {};
  \node (s) at (11.5,5) {};
  \node (x) at (20,2) {};
\node (y) at (20,4) {};
\node (z) at (20,6) {};

\foreach \from/\to in {o/p,o/r,o/q,p/r,p/q,q/r,s/p,r/x,r/y,r/z,y/z,t/p,t/r}
 \draw (\from) -- (\to);

\node (o1) at (29,2) {};
\node (p1) at (26,5)  {};
\node (q1) at (26,2)  {};
\node (r1) at (29,5) {};
 \node (t1) at (27.5,8) {};
 \node (s1) at (23,3) {};
\node (s2) at (23,6) {};
\node (y1) at (32.5,3) {};
\node (z1) at (32.5,6) {};

\foreach \from/\to in {o1/r1,o1/q1,o1/p1,p1/r1,p1/q1,q1/r1,s1/p1,s2/p1,r1/y1,r1/z1,t1/p1,t1/r1}
\draw (\from) -- (\to);

\draw[dashed] (s1) -- (s2) ;
 \draw[dashed] (y1) -- (z1) ;
\node at (24.5,4.5) [draw=none,fill=none] {$^{K_p}$};
\node at (31,4.5) [draw=none,fill=none] {$^{K_q}$};

\node at (16,0) [draw=none,fill=none] {$Pt_1$};
\node at (28,0) [draw=none,fill=none] {$Pt_2(p,q)$};

     \end{tikzpicture}
\caption{The graph $Pt_1$ and the triple block $Pt_2(p,q)$. }
 \label{fig:PT1-PT2}
\end{center}
\end{figure}

In Section~\ref{sec:spec-results} (Theorem \ref{the:Pt2}), we prove that the graph $Pt_1$
and every triple block graph $Pt_2(p,q)$ satisfy the condition
$\lambda_2 < -1/2$.

\begin{theorem}\label{theo:pto-mvs3}
Let $G$ be a   graph of diameter 3 with at least one $mvs$ of cardinality  2.
If $\lambda_2(G) < -1/2$ then
 $G$ is a connected induced  subgraph of  $Pt_1$ or of a $Pt_2(p,q)$.
\end{theorem}

\begin{proof}
By    Propositions \ref{lem:mvs-cardinal3} and \ref{lem:Ptolemaic},
if $\lambda_2(G) < -1/2$ then $G$ is a Ptolemaic graph for which  all the $mvs$  have cardinality 1 or 2.
By Proposition 6, 
 we know that the $mvs$ of cardinality 2 belong to a full house or a diamond.
By Proposition \ref{lem:mvs-cardinal}, for  all $mvs$ $S$ of $G$ with cardinality 2 it holds  $\mu(S) =1$. Also, being a  graph of  diameter 3, $G$ has an induced $P_4$, say $v_1v_2v_3v_4$.
Let suppose   that  $G$    has exactly  two $mvs$ of cardinality 2, $S_1$ and $S_2$.
Figure \ref{fig:diam3A} shows the  Ptolemaic graphs  with diameter 3 and minimum number of vertices  that obey all these conditions.
The graph (a) in Figure \ref{fig:diam3A} is the forbidden subgraph $F_{13}$; graphs (b) and (c) contain the forbidden subgraph $F_9$ and
 both the graphs (d) and (e) contain the forbidden subgraph $F_{10}$.
Hence, the graph $G$ cannot have more than one $mvs$ of cardinality 2.  Also, we may note that in graphs (d) and (e) of Figure \ref{fig:diam3A},  the $mvs$ are edges of  $P_4$.
Taking this fact into account, we conclude that  if the allowed $mvs$ is the edge
$\{v_1,v_2\}$, the graph contains the forbidden subgraph $F_{10}$.
So, the $mvs$ of cardinality 2 of  $G$  must be the edge $\{v_2,v_3\}$ of $P_4$.

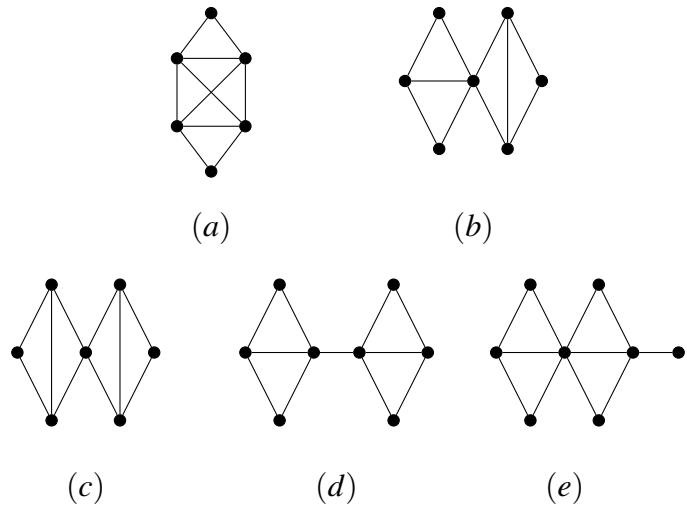
\begin{figure}[h]
\begin{center}
\begin{tikzpicture} [scale=.3,auto=left]
 \tikzstyle{every node}=[circle, draw, fill=black,
                        inner sep=1.5pt,  minimum width=4pt]

\node (x) at (-5,13) {};
 \node (y) at (-8,16)  {};
 \node (z) at (-8,13)  {};
\node (t) at (-5,16) {};
\node (u) at (-6.5,18) {};
\node (v) at (-6.5,11) {};

\node at (-6.5,8.5) [draw=none,fill=none] {$(a)$};

\foreach \from/\to in {x/y,x/z,x/t,x/v,y/u,y/z,y/t,z/t,z/v,t/u}
\draw (\from) -- (\to);

\node (a3) at (2,15) {};
\node (b3) at (5,15)  {};
\node (c3) at (8,15)  {};
\node (d3) at (3.5,18) {};
\node (e3) at (3.5,12) {};
\node (f3) at (6.5,12) {};
\node(g3) at (6.5,18){};

\node at (5,8.5) [draw=none,fill=none] {$(b)$};

\foreach \from/\to in {a3/b3,a3/d3,a3/e3,b3/d3,b3/e3,b3/g3,b3/f3,c3/g3,c3/f3,f3/g3}
\draw (\from) -- (\to);


\node (a5) at (-9,3) {};
\node (b5) at (-12,3)  {};
\node (c5) at (-15,3)  {};
\node (d5) at (-10.5,6) {};
\node (e5) at (-10.5,0) {};
\node (f5) at (-13.5,0) {};
\node(g5) at (-13.5,6){};

\node at (-12,-3) [draw=none,fill=none] {$(c)$};

\foreach \from/\to in {a5/d5,a5/e5,b5/d5,b5/e5,b5/g5,b5/f5,c5/g5,c5/f5,d5/e5,f5/g5}
\draw (\from) -- (\to);

\node (a2) at (-3.5,0) {};
\node (b2) at (-5,3)  {};
\node (d2) at (-2,3) {};
\node(d21) at (0,3) {};
\node (e2) at (-3.5,6) {};
\node (f2) at (3,3) {};
\node(g2) at (1.5,0){};
\node(h2) at (1.5,6){};

\foreach \from/\to in {a2/b2,a2/d2,b2/d2,e2/d2,b2/e2,f2/d21,h2/d21,g2/d21,h2/f2,g2/f2,d2/d21}
\draw (\from) -- (\to);

\node at (-1,-3) [draw=none,fill=none] {$(d)$};

\node (a4) at (7.5,0) {};
\node (b4) at (6,3)  {};
\node (d4) at (9,3) {};
\node(d41) at (12,3) {};
\node (e4) at (7.5,6) {};
\node (f4) at (14,3) {};
\node(g4) at (10.5,0){};
\node(h4) at (10.5,6){};

\foreach \from/\to in {a4/b4,a4/d4,b4/d4,b4/e4,d4/e4,d4/d41,d4/h4,d4/g4,d41/f4,
d41/h4,d41/g4}
\draw (\from) -- (\to);

\node at (9,-3) [draw=none,fill=none] {$(e)$};

 \end{tikzpicture}
\caption{Smallest Ptolemaic graphs with diameter $>2$ and two $mvs$ of cardinality 2.}
\label{fig:diam3A}
\end{center}
\end{figure}

 Let     $G$ be a graph with  $\mathrm{diam}(G)=3$ and $\lambda_2 <-1/2$ that has  one $mvs$ of cardinality 2 belonging to a full house or a diamond  corresponding to
  the edge $\{v_2,v_3\}$ of the induced $P_4$.  Then there are two  $mvs$, $S_i$ and $S_j$, of cardinality 1  corresponding to vertices $v_2, v_3$. If each one of these $mvs$ have multiplicity 1 then $G$ is a   triple block graph $Pt_2(p,q)$, for $p, q \geq 2$.

  Let us now consider that the two $mvs$ of cardinality 1  of $G$ satisfy $\mu(S_i) =1$ and $\mu(S_j)=2$.
The graph in Figure \ref{fig:diam3B}(a) is a connected induced subgraph of the split graph $SP_1$, which satisfies  $\lambda_2 <-1/2$ (Theorem \ref{theo:split}).
By adding one vertex to this graph, maintaining the multiplicities of the $mvs$,
 we obtain  graphs (b) and (d) in Figure \ref{fig:diam3B}.
The graph (b) has $\lambda_2 \approx -0.50229 <-1/2$ but  the same is not true for  graph (d)  because it has    $F_8$ as a subgraph.
Adding one more vertex to graph (b) in  Figure \ref{fig:diam3B},  we obtain graph (c), which has $\lambda_2 \approx -0.49839$, so it does not satisfy the required condition on $\lambda_2$.
 Finally, Figure \ref{fig:diam3C} shows two  Ptolemaic graphs with diameter 3 and $mvs$ $S_i$ and $S_j$ of cardinality 1:
 graph (a) is the smallest one
such that $\mu(S_i) =1$ and $\mu(S_j) =3$, and
 graph (b) is the smallest  with   $\mu(S_i) =\mu(S_j) =2$.
Both are split graphs but,  by Theorem \ref{theo:split}, they do not obey the condition $\lambda_2 <-1/2$.
We conclude that, if our graph $G$ satisfies $\lambda_2 <-1/2$,  $\mathrm{diam}(G)=3$ and it is not a block graph  then  it is a connected induced subgraph of $Pt_1$ or of a triple block graph $Pt_2(p,q)$, $p,q \geq 2$.
\end{proof}


\begin{figure}[h]
\begin{center}
\begin{tikzpicture} [scale=.3,auto=left]
 \tikzstyle{every node}=[circle, draw, fill=black,
                        inner sep=1.5pt, minimum width=4pt]

\node (a1) at (1,10) {};
 \node (b1) at (3,10)  {};
\node (c1) at (6,10)  {};
\node (d1) at (8,10) {};
 \node (e1) at (7,12) {};
\node (f1) at (4.5,13) {};
\node (g1) at (4.5,7) {};

\node at (4.5,4) [draw=none,fill=none] {$(a)$};

\foreach \from/\to in {a1/b1,b1/c1,b1/f1,b1/g1,c1/d1,c1/e1,c1/f1,c1/g1}
\draw (\from) -- (\to);

\node (a2) at (11,10) {};
 \node (b2) at (13,10)  {};
\node (c2) at (16,10)  {};
\node (d2) at (18,10) {};
 \node (e2) at (17,12) {};
\node (f2) at (14.5,13) {};
\node (g2) at (14.5,7) {};
\node (h2) at (17,8) {};

\node at (14.5,4) [draw=none,fill=none] {$(b)$};

\foreach \from/\to in {a2/b2,b2/c2,b2/f2,b2/g2,c2/d2,c2/e2,c2/f2,c2/g2,c2/h2,d2/h2}
\draw (\from) -- (\to);

\node (a3) at (21,10) {};
 \node (b3) at (23,10)  {};
\node (c3) at (26,10)  {};
\node (d3) at (28,10) {};
 \node (e3) at (27,12) {};
\node (f3) at (24.5,13) {};
\node (g3) at (24.5,7) {};
\node (h3) at (27,8) {};
\node (i3) at (29,8.5) {};

\node at (24,4) [draw=none,fill=none] {$(c)$};

\foreach \from/\to in {a3/b3,b3/c3,b3/f3,b3/g3,c3/d3,c3/e3,c3/f3,c3/g3,c3/h3,d3/h3,
i3/c3,i3/d3,i3/h3}
\draw (\from) -- (\to);

\node (a4) at (31,11) {};
 \node (b4) at (33,10)  {};
\node (c4) at (36,10)  {};
\node (d4) at (38,10) {};
 \node (e4) at (37,12) {};
\node (f4) at (34.5,13) {};
\node (g4) at (34.5,7) {};
\node (h4) at (31,9) {};

\node at (36,4) [draw=none,fill=none] {$(d)$};

\foreach \from/\to in {a4/b4,a4/h4,b4/c4,b4/f4,b4/g4,b4/h4,c4/d4,c4/e4,c4/f4,c4/g4}
\draw (\from) -- (\to);

 \end{tikzpicture}
\caption{Smallest graphs of diameter 3  with two $mvs$ of cardinality 1, one  of multiplicity 1 and one with multiplicity 2.}
\label{fig:diam3B}
\end{center}
\end{figure}
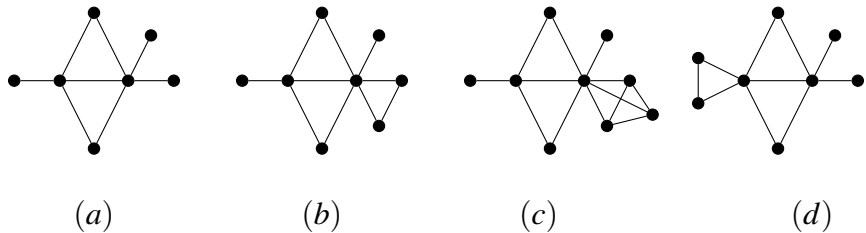

\begin{figure}[h]
\begin{center}
\begin{tikzpicture} [scale=.3,auto=left]
 \tikzstyle{every node}=[circle, draw, fill=black,
                        inner sep=1.5pt, minimum width=4pt]

\node (a1) at (1,10) {};
 \node (b1) at (3,10)  {};
\node (c1) at (6,10)  {};
\node (d1) at (8,10) {};
 \node (e1) at (7,12) {};
\node (f1) at (4.5,13) {};
\node (g1) at (4.5,7) {};
\node (h1) at (7,8) {};

\node at (4.5,4) [draw=none,fill=none] {$(a)$};

\foreach \from/\to in {a1/b1,b1/c1,b1/f1,b1/g1,c1/d1,c1/e1,c1/f1,c1/g1,c1/h1}
\draw (\from) -- (\to);

\node (a2) at (11,10) {};
 \node (b2) at (13,10)  {};
\node (c2) at (16,10)  {};
\node (d2) at (18,10) {};
 \node (e2) at (17,12) {};
\node (f2) at (14.5,13) {};
\node (g2) at (14.5,7) {};
\node (h2) at (11,12) {};

\node at (14.5,4) [draw=none,fill=none] {$(b)$};

\foreach \from/\to in {a2/b2,b2/c2,b2/f2,b2/g2,c2/d2,c2/e2,c2/f2,c2/g2,b2/h2}
\draw (\from) -- (\to);

 \end{tikzpicture}
\caption{Smallest graphs with  $mvs$ of different multiplicities.}
\label{fig:diam3C}\end{center}
\end{figure}
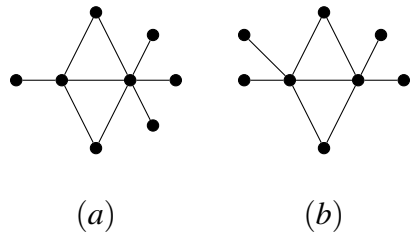

\begin{proposition}\label{prop:diameter4}
Let $G$ be a chordal graph with at least one $mvs$ of cardinality 2.
If $\lambda_2(G) < -1/2$ then the  diameter of $G$ is at most equal to 3.
\end{proposition}
\begin{proof}
Let $\mathbb S$ be the set of minimal vertex separators of $G$.
Let $S \in \mathbb S$ be a $mvs$ with cardinality 2;
since the adjacent vertices of $S$ must be simplicial vertices (Lemma \ref{lem:full-diamond}),
$S$ must be an edge of an induced path $v_1v_2v_3v_4v_5$.
In this case  $G$ contains the forbidden subgraph $F_{10}$ and then, it has $\lambda_2(G) > -1/2$, a contradiction.\
\end{proof}

\section{Spectral  results on graphs satisfying $\lambda_2<-1/2$}
\label{sec:spec-results}

In the previous section, we described the structures of graphs with diameters 2 and 3 satisfying $\lambda_2<-1/2$ and possessing at least one $mvs$ of cardinality 2, which, therefore, are not block graphs. For the latter, Lemma \ref{lem:block} provides a complete characterization of those that satisfy the required condition on $\lambda_2$. In our constructions, which focused, in addition to the diameter, on aspects of the minimal vertex separators presented in the structure of graphs satisfying the above spectral condition, we introduced the relaxed block star graphs  and also the (Ptolemaic)  graphs $Pt_1$ and triple blocks $Pt_2(p,q)$, where $p,q\geq 2$. Now, this section is dedicated  to showing that these graphs actually satisfy the condition $\lambda_2<-1/2$. We begin by recalling other spectral properties of the distance matrix of graphs.

 \begin{lemma}(Lemmas 3.4 and 3.5 \cite{LuHuangHuang2017}\label{lem:lemasLuHH})
  Let $G=(V,E)$ be a  graph on $n$ vertices. It holds that:
  \begin{enumerate}
      \item[(i)] if $S = \{v_1, \ldots , v_r\} \subset V$ $ (r \geq  2)$
induces a clique of $G$ with $N_G(v_i )\setminus S = N_G(v_j ) \setminus S$ for  $1 \leq  i, j \leq p$  then $-1$ is an
eigenvalue of $\mathbf{D}(G)$ with multiplicity at least $r-1$.

\item[(ii)] if $S = mK_r$ $(m \geq  2)$ is an
induced subgraph of $G$ with $N_G(u) \setminus V(S) = N_G(v) \setminus V(S)$ for any vertices  $u, v \in V(S)$, then
$-(r + 1)$ is an eigenvalue of $\mathbf{D}(G)$ with multiplicity at least $m -1$.
 \end{enumerate}
 \end{lemma}

 The theory of equitable partitions for symmetric matrices can be found in \cite{CRS2010}. The particular case of the distance matrix is studied in detail in \cite{LuHuangHuang2017}. Here,  we recall the main concepts and results of this theory, for further reference.

Denote by $d(v, S) = \sum_{u\in S} d_G(u, v)$, where $v \in V$ and $S$ is a non-empty subset
of $V$, the vertex set of a given graph $G$.
The  partition  $\pi=\{ V_1,
V_2, \ldots , V_k\}$ of $V$ is called a \textit{distance equitable partition}  if, for any $v \in V_i$, $d(v, V_j ) = b_{ij}$
is a constant depending only on $i, j$, for  $1 \leq i, j \leq k$. Here the matrix $\mathbf{F}_\pi(G) = [b_{ij} ]_{k\times k}$ is
called the \textit{distance divisor matrix} of $G$ with respect to $\pi$.

\begin{lemma}\cite{LuHuangHuang2017}\label{lem:equit}
Let $G=(V,E)$ be a  graph with distance matrix  $\mathbf{D}$, and let $\pi= \left\{
V_1, V_2,  \ldots \right. $ $\left.\ldots,  V_k \right\}$ be a distance equitable partition of $V$ with distance divisor matrix
$\mathbf{F}_\pi$.
 Then $\det(x \mathbf{I} - \mathbf{F}_\pi)$ divides
 $\det(x\mathbf{I} - \mathbf{D})$. Furthermore, $\lambda_1(G)$ is an eigenvalue of $\mathbf{F}_\pi$ and also, the largest one.
\end{lemma}

\section*{Relaxed block stars satisfy $\lambda_2 <-1/2$ } \label{sec:main}

In Section~\ref{sec:4-struct}, in the study of Ptolemaic graphs of diameter~2
satisfying $\lambda_2 < -1/2$ (see Theorem~\ref{theo:pto-mvs2}), we introduced
the class of graphs that we call \emph{relaxed block stars}. We now aim to show
that these graphs do indeed satisfy the spectral condition on $\lambda_2$.
In this way, the characterization of graphs of diameter~2 whose distance matrix
has second largest eigenvalue less than $-1/2$ will be complete,
see Corollary~\ref{cor:caract-diam2}.

To achieve this goal, we apply Descartes' Rule of Signs and Sturm's Theorem,
which will be recalled in due course.

First, we observe that relaxed block star graphs can be described as a specific
join construction, as follows.   Indeed, a relaxed block star is a graph
of the form
\begin{equation}\label{eq:geral-relax}
K_1 \oplus \left(
(p_1K_{r_1} \cup \cdots \cup p_kK_{r_k})
\cup q\bigl(K_1 \oplus (K_1 \cup K_2)\bigr)
\cup s\, P_3
\right),
\end{equation}
for any positive integers
$k, p_1,\ldots,p_k, r_1,\ldots,r_k, q,$ and $s$.

For example, the graph in Figure~\ref{fig:examplerelaxed} is
$K_1 \oplus \bigl((K_1 \cup K_4) \cup (K_1 \oplus (K_1 \cup K_2)) \cup P_3\bigr)$.
Note that $P_3$ is an induced subgraph of $K_1 \oplus (K_1 \cup K_2)$.
Moreover, for $1 \le i \le k$, each $K_{r_i}$ is an induced subgraph of $K_r$,
where $r=\max_{1 \le i \le k} r_i$. Taking into account
Remark~\ref{rem:interl}, the fact that any relaxed block star satisfies
$\lambda_2 < -1/2$ is therefore a straightforward consequence of the following
theorem.

\begin{theorem}\label{the:main}
Let
\begin{equation}\label{eq:Grpq}
G = G(r,p,q) = K_1 \oplus \bigl(pK_r \cup q(K_1 \oplus (K_1 \cup K_2))\bigr),
\end{equation}
where $r,p,q \ge 1$ are positive integers. Then the second largest
$\mathbf{D}$-eigenvalue of $G(r,p,q)$ is less than $-1/2$.
\end{theorem}

Thus, our goal in this section is to prove Theorem~\ref{the:main}. To this
end, we first establish several auxiliary results, which will be stated and
proved as lemmas. Throughout the remainder of this subsection, we assume that
$G=G(r,p,q)$ is a fixed relaxed block star as defined in
(\ref{eq:Grpq}), with parameters $p,q,r \ge 1$ and characteristic polynomial
$P_G(x)$.
\smallskip

\begin{lemma}\label{fator1e2} For the   characteristic polynomial $P_G(x)$ of the graph $G=G(r,p,q)$, it holds that:
\begin{enumerate}
     \item[(i)] assuming that  $r\geq 2$, 
     then  $(x+1)^{(p(r-1)+q)}$ divides $P_G(x)$;
     \item[(ii)] if  $p\geq 2$ then $(x+(r+1))^{(p-1)}$ divides $P_G(x)$.
\end{enumerate}
\end{lemma}

\begin{proof}  Indeed,  that $-1$ is an eigenvalue of $G$ with multiplicity at least $p(r-1) + q$ is evidenced by the existence of at least $p(r-1) + q \geq 2$ vertex subsets  satisfying condition (i) of Lemma \ref{lem:lemasLuHH},  $r-1$ of which corresponding to  one of the $p$ copies of $K_r$ and  one associated with each instance of  $K_2$ in the  $q $ copies of $K_1 \oplus (K_1 \cup K_2)$ in the structure of $G$.
Analogously, by condition (ii) of Lemma \ref{lem:lemasLuHH},  it happens that $-(r+1)$ is an eigenvalue of $G$ with multiplicity at least $p-1$, as ensured by considering the $p \geq 2$ copies of $K_r$.
    \end{proof}
\smallskip

 In our computational tests, we found a factor of $P_G(x)$
 associated with the presence of at least two copies of the graph \( K_1 \oplus (K_1 \cup K_2) \) as connected induced subgraphs of $G$. The next lemma presents the arguments for this fact.

\begin{lemma}\label{fator3} With the notations of Theorem \ref{the:main}, the polynomial $(x^3+7x^2+13x+5)^{q-1}$ divides $P_G(x)$ when $q \geq 2$.
\end{lemma}

\begin{proof}

The proof becomes simpler if  we assign labels to the vertices of the graph $G=G(r,p,q)$ as follows:
 let \( w_0 \) denote the universal vertex of $G$ 
 and,  for each \( i = 1, \dots, p \), let  the vertices \( w_{(i-1)r+1}, \dots, w_{ir} \) form the \( i \)-th copy of \( K_r \), with the final one given by \( w_{(p-1)r+1}, \dots, w_{pr} \).

We now label the vertices of each copy of \( K_1 \oplus (K_1 \cup K_2) \) as follows: let \( v_1 \) be the universal vertex, \( v_2 \) the isolated vertex, and \( v_3, v_4 \) the vertices of the \( K_2 \) component. We repeat this labeling for each one of the \( q \) copies. Specifically, for each \( i = 1, \dots,  q \), the vertices \( v_{4(i-1)+1}, v_{4(i-1)+2}, v_{4(i-1)+3}, v_{4(i-1)+4} \) correspond to the \( i \)-th copy of \( K_1 \oplus (K_1 \cup K_2) \), where \( v_{4i+1} \) is the universal vertex, \( v_{4i+2} \) is the isolated vertex, and \( v_{4i+3}, v_{4i+4} \) are the vertices of the \( K_2 \) component.
Using this labeling, the matrix $\mathbf{D}(G)$ can be written as

\[\mathbf{D}(G)=\left[
   \begin{array}{c|c}
     \mathbf{D}_1 & \mathbf{D}_2 \\
     \hline
     \mathbf{D}_2^t & \mathbf{D}_3 \\
   \end{array}
 \right]\,,\]
where $\mathbf{D}_1$ has order $(pr+1)\times (pr+1)$, $\mathbf{D}_2$ has order $(pr+1)\times 4q$ and finally $\mathbf{D}_3$ has order $4q\times 4q$.
Let's describe the matrices $\mathbf{D}_2$ and $\mathbf{D}_3$  in more detail. In fact,
\[
\mathbf{D}_2=\left[
\begin{array}{cccc}
1&1&\dots&1\\
2&2&\dots&2\\
\vdots&\vdots&\ddots&\vdots\\
2&2&\dots&2\\
\end{array}
\right],
\]
since its entries are equal to the distances $d(w_j, v_i).$
The matrix \( \mathbf{D}_3 \) encodes the pairwise distances between the vertices \(v_1, \dots, v_{4q}\).
This matrix can naturally be viewed as a block matrix, where each block corresponds to the distances between specific subsets of vertices.
In fact,
\[
\mathbf{D}_3=\left[
\begin{array}{cccc|cccc|c|cccc}
0&1&1&1&2&2&2&2&\dots&2&2&2&2\\
1&0&2&2&2&2&2&2&\dots&2&2&2&2\\
1&2&0&1&2&2&2&2&\dots&2&2&2&2\\
1&2&1&0&2&2&2&2&\dots&2&2&2&2\\
\hline \
2&2&2&2&0&1&1&1&\dots&2&2&2&2\\
2&2&2&2&1&0&2&2&\dots&2&2&2&2\\
2&2&2&2&1&2&0&1&\dots&2&2&2&2\\
2&2&2&2&1&2&1&0&\dots&2&2&2&2\\
\hline\
&&\dots&&&\dots&&&\dots&&&\dots\\
\hline\
2&2&2&2&2&2&2&2&\dots&0&1&1&1\\
2&2&2&2&2&2&2&2&\dots&1&0&2&2\\
2&2&2&2&2&2&2&2&\dots&1&2&0&1\\
2&2&2&2&2&2&2&2&\dots&1&2&1&0\\
\end{array}
\right].
\]

We claim that the eigenvalues of  matrix
 $$\mathbf{M}=\begin{bmatrix}
-2&-1&-2\\
-1&-2&\hfill 0\\
-1&\hfill 0&-3
 \end{bmatrix}$$
 are eigenvalues of $\mathbf{D}(G).$

In order to prove our claim,  let \( \lambda \) be an eigenvalue of the matrix \( \mathbf{M} \), and let \( [x_\lambda, y_\lambda, z_\lambda]^{t} \) denote an eigenvector associated with \( \lambda \).
Then the $q-1$ vectors
\begin{align*}
    \mathbf{u}_1&=[\underbrace{0, \dots, 0}_{pr+1}, x_\lambda, y_\lambda, z_\lambda, z_\lambda, -x_\lambda, -y_\lambda, -z_\lambda, -z_\lambda,\underbrace{0, \dots, 0}_{4(q-2)}]^{t}\\
   \mathbf{u}_2&=[\underbrace{0, \dots, 0}_{pr+1}, x_\lambda, y_\lambda, z_\lambda, z_\lambda, 0, 0, 0, 0, -x_\lambda, -y_\lambda, -z_\lambda, -z_\lambda,\underbrace{0, \dots, 0}_{4(q-3)}]^{t}\\
   \vdots \\
    \mathbf{u}_{q-1}&=[\underbrace{0, \dots, 0}_{pr+1}, x_\lambda, y_\lambda, z_\lambda, z_\lambda, \underbrace{0, \dots, 0}_{4(q-2)}, -x_\lambda, -y_\lambda, -z_\lambda, -z_\lambda]^{t}\\
\end{align*}
are eigenvectors of $\mathbf{D}(G)$ associated with $\lambda.$
We  prove, for instance, that \( \mathbf{u}_1 \) is an eigenvector of \( \mathbf{D}(G) \) associated with the eigenvalue \( \lambda \):
$$\mathbf{D}(G)\mathbf{u}_1=\left[
\begin{array}{c|c}
\mathbf{D}_1&\mathbf{D}_2\\
\hline
\mathbf{D}_2^t&\mathbf{D}_3
\end{array}
\right]\mathbf{u}_1=\left[\begin{array}{c}0\\ \vdots\\0\\-2x_\lambda-y_\lambda-2z_\lambda\\
-x_\lambda-2y_\lambda\\-x_\lambda-3z_\lambda\\-x_\lambda-3z_\lambda\\2x_\lambda+y_\lambda+2z_\lambda\\
x_\lambda+2y_\lambda\\x_\lambda+3z_\lambda\\x_\lambda+3z_\lambda\\0\\ \vdots\\0
\end{array}\right]=\left[\begin{array}{c}0\\ \vdots\\0\\\hfill \lambda\cdot x_\lambda\\
\hfill\lambda\cdot y_\lambda\\\hfill\lambda\cdot z_\lambda\\\hfill\lambda\cdot z_\lambda\\-\lambda\cdot x_\lambda\\
 -\lambda\cdot y_\lambda\\-\lambda\cdot z_\lambda\\-\lambda\cdot z_\lambda\\0\\ \vdots\\0
\end{array}\right]=
\lambda\cdot \mathbf{u}_1,
$$
where the  equalities  follows from the fact that $[x_\lambda, y_\lambda, z_\lambda]^{t}$ is an eigenvector of $\mathbf{M}$ associated with the eigenvalue $\lambda$.
Given that the characteristic polynomial of $\mathbf{M}$ is $ P_M(x) = x^3 + 7x^2 + 13x + 5 $, and that each one of its roots appears as eigenvalue of  $\mathbf{D}(G)$  with multiplicity at least \( q - 1 \), the lemma follows immediately.
\end{proof}
\smallskip

In order to prove that $\lambda_2(G) <-1/2$ holds for our $G=G(r,p,q)$, we consider now  the partition $\pi = \left\{V_1, V_2, V_3, V_4, V_5\right\}$ of the vertex set  $V$ of $G$,  where
  $V_1$ consists of its universal vertex
 and
 $V_2$ comprises the collection of vertices corresponding to the $p$ copies of the complete graph $K_r$; and also,
in order to describe the vertex set of the subgraphs $q(K_1 \oplus (K_1 \cup K_2))$, we proceed by setting $V_3$ and $V_4$ as the sets containing, respectively,  the universal vertices and the independent vertices associated with each instance of $K_1$ in the $q$ copies of $K_1 \oplus (K_1 \cup K_2)$,
and $V_5$ as the set formed by the pairs of vertices corresponding to the subgraphs isomorphic to $K_2$ within each of the $q$ copies of $K_1 \oplus (K_1 \cup K_2)$.
We may observe that $\#V_1=1,\; \#V_2= pr,\;  \#V_3=q,\; \#V_4=q$ and $\#V_5=2q.$
\smallskip

Clearly, $\pi$ is a distance  equitable partition of $V$. The distance divisor matrix of $G$ associated with $\pi$ is

\medskip

\[ \mathbf{F}_{\pi}(G)=\left[
   \begin{array}{ccccc}
     0 & p r & q & q & 2 q\\
1 & r-1+2\left( p-1\right)  r & 2 q & 2 q & 4 q\\
1 & 2 p r & 2 \left( q-1\right)  & 1+2 \left( q-1\right)  & 2+4(q-1)\\
1 & 2 p r & 1+2 \left( q-1\right)  & 2 \left( q-1\right)  & 4 q\\
1 & 2 p r & 1+2 \left( q-1\right)  & 2 q & 1+2 \left( 2 q-2\right) \\
   \end{array}
 \right]\,.\]

\smallskip

Its characteristic polynomial is
 \[f(r,p,q, x) ={{x}^{5}}-\left( 8 q+2 r p-r-8\right) \, {{x}^{4}}
 -\left( \left( 8 r+36\right)  q+15 r p-7 r-20\right) \, {{x}^{3}}+\]
 \[- \left( \left( 28 r+52\right) q+33 r p-13 r-18\right) \, {{x}^{2}}-
 \left( \left( 24 r+30\right)  q+23 r p-5 r-5\right)x+\]
 \[-\left(\left( 6 r+6\right)  q+5 r p\right).
 \]
\smallskip

We  write $f(x)$ instead of $f(r,p,q,x)$ for brief, when there is no chance of misunderstanding.  From Lemma \ref{lem:equit}, we  then have the following result.

 \begin{lemma}\label{fator4} The above polynomial $f(x)$  divides the characteristic polynomial $P_G(x)$ of $G$.
 \end{lemma}
\smallskip

 \begin{proposition} The polynomial $P_G(x)$ can be factored as
 \[P_G(x)=(x+1)^{(p(r-1)+q)}\;(x+(r+1))^{(p-1)}\;({{x}^{3}}+
7{{x}^{2}}+13x+5)^{(q-1)}\;f(r,p,q,x)\,.\]
Furthermore, the roots of $(x+1)$, $(x+(r+1))$ and $(x^3+7x^2+13x+5)$ are less than $-1/2$.
 \end{proposition}
\begin{proof}
    The proof follows straightforwardly from Lemmas \ref{fator1e2}(i) and (ii), \ref{fator3} and \ref{fator4}, considering that factors $(x+1)$, \ $(x+(r+1))$\  and \ $(x^3+7x^2+5)$ \ occur provided the para\-meters $r,p,q$ satisfy the conditions indicated in the related results.  Moreover,  the roots of $(x+1)$ and $(x+(r+1))$ are clearly less that $-1/2$.  Similarly, the roots of the cubic polynomial  $({{x}^{3}}+
7{{x}^{2}}+13x+5)$ are   $x \approx -0.51881$, $x \approx -2.3111$\, and  $x \approx -4.1701$, being all of which also less than $-1/2$.
\end{proof}
\medskip

To complete the proof of Theorem \ref{the:main}, it remains to deal with the polynomial $f(x)=f(r,p,q,x)$ of matrix $\mathbf{F}_{\pi}(G) $. In order to make it simpler, we adopt the following notations:
\[A={A}\left( r,p,q\right) :=8 q+2 r p-r-8;\]
\[B= {B}\left( r,p,q\right) :=\left( 8 r+36\right)  q+15 r p+\left( -7\right)  r-20;\]
\[C= {C}\left( r,p,q\right) :=\left( 28 r+52\right)  q+33 r p+\left( -13\right)  r-18;\]
\[D= {D}\left( r,p,q\right) :=\left( 24 r+30\right)  q+23 r p+\left( -5\right)  r-5;\]
\[E= {E}\left( r,p,q\right) :=\left( 6 r+6\right)  q+5 r p.\]
Thus we can refer to  the characteristic polynomial of $\mathbf{F}_{\pi}(G)$ as
\[{f(r,p,q,x)}={{x}^{5}}-{A}\left( r,p,q\right) \, {{x}^{4}}-{B}\left( r,p,q\right) \, {{x}^{3}}-{C}\left( r,p,q\right) \, {{x}^{2}}-{D}\left( r,p,q\right)  x-{E}\left( r,p,q\right), \]
or briefly, as
\[f(x)=f\left( A,B,C,D,E,x\right)
={{x}^{5}}-A\, {{x}^{4}}-B\, {{x}^{3}}-C\, {{x}^{2}}-D x-E.\]
\medskip

In order to prove that, with exception of the largest one,  all roots  of
  our real   degree 5 polynomial \( f(x) \)   are less than -1/2, we   show that $f(x)$  has no roots in the interval $[-1/2, 0]$. To get this, we  prove the following sequence of facts:
\begin{itemize}
\item[(I)] All the roots of $f(x)$, except the largest one, are negatives;
    \item[(II)] \( f(-1/2) < 0 \) and \( f(0) < 0 \);
    \item[(III)] The second derivative \( f''(x) \) of $f(x)$ does not vanish on \( [-1/2, 0] \); specifically,  \( f''(x) <0 \)   on this interval.
\end{itemize}
 From fact (III),  it follows that  the first derivative \( f'(x) \) is a monotonically decreasing function on the interval \( [-1/2, 0] \). Also, in particular, it holds that  $0> f'(-1/2) >f'(0)$.
    Therefore, \( f(x) \) is  strictly concave throughout \( [-1/2, 0] \).
 Since $f'(x)$ does not vanish in this interval, the maximum of \( f(x) \) on the interval is attained at one of the endpoints.
 Given that both \( f(-1/2) < 0 \) and \( f(0) < 0 \) from (II),  it follows that \( f(x) < 0 \) for all \( x \in [-1/2, 0] \).
So, we can conclude that  \( f(x) \) has no real roots in the interval \( [-1/2, 0] \).
\smallskip

For proving Fact (I), we  apply
 the next result, known as Descartes' Rule of Signs, which gives the maximum number  of positive and of negative roots of an univariate polynomial. For more details, see, for example,  \cite{Anderson01051998}, by Anderson \textit{et al.} and \cite{Wang2004}, by Wang.

\begin{lemma}[Descartes' Rule of Signs, 1637]\label{lem:Desc}
 Let \[p(x) = a_0 x^{b_0} + a_1 x^{b_1} +\ldots +a_nx^{b_n}\] denote
a polynomial with  real coefficients $a_i$, where the $b_i$ are integers satisfying
$0 <  b_0  <  b_2 < \ldots < b_n$.
Let denote by $\#_+(p(x))$ the number of positive real roots of $p(x)$ counted with their multiplicities and by $\nu(p(x))$ the number of sign changes between consecutive non-zero of its coefficients $a_0, \ldots , a_n$.
Then \[\#_+(p(x))-\nu(p(x)) \mbox{ is a non-negative even integer.}\]
In particular, if $\nu(p(x)) \leq 1$ then one has $\#_+(p(x)) =\nu(p(x))$.
Also, for  the number $\#_-(p(x))$
of negative zeros of $p(x)$ (counted with multiplicities)
it holds that \[\#_-(p(x)) -\nu(p(-x)) \mbox{ is a non-negative even integer.}\]
\end{lemma}
\smallskip

The veracity of Fact (III) will be proven as a consequence of Sturm's Theorem, which we state here as presented in Biagioli, see \cite{biagioli2016methods}. Before this, however, we need to introduce the  concepts of
\textit{Sturm sequence} and \textit{sign variation}. Given a univariate polynomial $g(x)$ of degree $n$, consider the polynomials
\[g_0(x) = g(x),\]  \[g_1(x) = g^{\prime}(x), \mbox {the derivative of g, }\]  and, for $i >1$,
\[g_i(x) = (-1) \times \mbox{[the
remainder of the division of $g_{i-2}(x)$ by $g_{i-1}(x)$ ]}. \]
The sequence \[(g_0(x),\ g_1(x), \ldots,\ g_{n-1}(x),\ g_n(x))\] is called \textit{the Sturm sequence} of $g$.
Clearly,  the length of the Sturm sequence is at most the degree of $g(x)$.
Also, the \textit{number of sign variations} $\omega(y)$   of the Sturm sequence of $g(x)$ at $y \in \mathbb{R}$ is the number of sign changes (ignoring zeros) in the sequence of real numbers
\[(g_0(y),\ g_1(y), \ldots,\ g_{n-1}(y),\ g_n(y)).\]

 \begin{lemma}[Sturm's Theorem, 1829]\label{lem:Sturm}  The number of roots that the polynomial $g(x)$ has in the interval
$(a, b)$, provided that $g(a) \neq 0$, $g(b) \neq 0$ and $a < b$, is equal to $\omega(b) - \omega(a)$.
      \end{lemma}
\smallskip

Finally, we are able to present the proof of the main theorem of this subsection.
\smallskip

\begin{proof} (of Theorem \ref{the:main}) Firstly,   we claim  that Fact (I) is true.
Indeed, we may note that all the coefficients of \[f(x)=f\left( A,B,C,D,E,x\right)
={{x}^{5}}-A\, {{x}^{4}}-B\, {{x}^{3}}-C\, {{x}^{2}}-D x-E\] are positive, due the conditions on the parameters $p,q,r$.
The sequence of signs between consecutive coeficients of $f(x)$  is
\[(+,-,-,-,-,-), \]
that is, $f(x)$  presents only one sign change. Then Descartes' Rule says that
\[\#_+(f(x))=\nu(f(x)). \]
Since  we know that $\lambda_1(G) >0$ is an eigenvalue of $\mathbf{F}_{\pi}(G)$ by Lemma \ref{lem:equit},  we deduce that it  is the only positive root of   $f(x)$.
\smallskip

Also, it is easy to verify that\  $f(0)=-\left(6q(a+1)+5ap\right) <0$ \ and \ that  $f(-1/2)= -{(2 a+1)}/{32} <0$; thus, Fact (II) is true.

For proving Fact (III), we  apply Sturm's Theorem  to the second derivative  of $f(x)= f( A,B,C,D,E,x)$ as a function of $x$, which we call $g(x)$; thus, let
\[g(x)=\operatorname{g}\left( A,B,C,D,E,x\right) := 20\,{{x}^{3}}-12 A\, {{x}^{2}}-6 B x-2 C\,.\]
We must construct the Sturm sequence of $g$ and show that $g(x)$ has no roots in the interval $[-1/2,0]$.
For this purpose, let write  $g_0(x)=g(x)$ and let
$g_1\left(x\right)$ be the derivative of $g\left(x\right)$ as a function of $x$, that is,
\[\operatorname{g_1}\left(x\right) :=60 \, {{x}^{2}}-24 A x-6 B.\]
Then, since
\[g_2(x) = (-1) \times \mbox{ [the
remainder of the division of $g_{0}(x)$ by $g_{1}(x)$]}, \] we have
\[\operatorname{g_2}\left(x\right) :=\frac{\left( 20 B+8 {{A}^{2}}\right)  x+10 C+2 A B}{5}\]
Also, considering that
\[g_3(x) = (-1) \times \mbox{[the
remainder of the division of $g_{1}(x)$ by $g_{2}(x)$]}, \] we have
\[\operatorname{g_3}\left(x\right) :=-\frac{375 {{C}^{2}}+\left( 450 A B+120 {{A}^{3}}\right)  C-150 {{B}^{3}}-45 {{A}^{2}}\, {{B}^{2}}}{25 {{B}^{2}}+20 {{A}^{2}} B+4 {{A}^{4}}}\,.\]
Let obtain the sign variations $\omega(0)$ and  $\omega(-1/2)$ of the Sturm sequence of $g(x)$ in the ends of the interval $[-1/2, 0]$:
\[\operatorname{g_0}\left( A,B,C,D,E,0\right)=-2C<0;\]
\[g_1\left( A,B,C,D,E,0\right)=-6B <0;\]
\[g_2\left( A,B,C,D,E,0\right)=\frac{10 C+2 A B}{5} >0\]
Thus the number of sign variations of   the Sturm sequence of $g(x)$ at $y=0$ is
\[\omega(0)=(-, -, + , *),\]
where the \ $*$ \ at the last position means that we don't have to worry about the signs of $g_3(0)$ and $g_3(-1/2)$: indeed,   $g_3(x)$ is a constant which do not depend on $x$.
In the case of the number of  variations $\omega(-1/2)$, we must analyze through the original parameters, meaning that we need to recover the dependence of $A, B, C, D, E$ on the initial parameters $r,p, q$. We have
\[g_0\left(r ,p,q,-1/2\right)=\frac{-4 C(r,p,q)+6 B(r,p,q)-6 A(r,p,q)-5}{2}=\]
\[=\frac{-64 r q-40 q-54 r p+16 r-5}{2}.\]
It is suffices to analyze the sign of the numerator:
\[-(64 r q+40 q+54 r p-16 r+5) = -(16r\underset{>0}{\underbrace{(4q-1)}}+40q+54rp+5)<0.\]
For $g_1(-1/2)$, since $p,q\geq 1$, we have
\[g_1\left( r,p,q,-1/2\right)= -6 B(r,p,q)+12 A(r,p,q)+15= \]
\[=-3r\left(16 q+22p-10\right)-3\left(40q-13\right)<0.\]
For $g_2(-1/2)$,  it holds that
\[g_2\left( A,B,C,D,E,-1/2\right) =\frac{10 C+\left( 2 A-10\right)  B-4 {{A}^{2}}}{5}=\frac{2\left(5(C-B)+A(B-2A)\right)}{5}. \]
Since $r,p,q \geq 1$, we have
\[C(r,p,q)-B(r,p,q)=20rq+16q+18rp-6r+2 >0 \  \ \ \mbox{ and  }\]
\[B(r,p,q)-2A(r,p,q) =\left( 8 r+20\right)  q+11 r p-5 r-4 >0.\]
Thus,
\[g_2\left( A,B,C,D,E,-1/2\right)=\frac{2\left(5(C-B)+A(B-2A)\right)}{5}>0\]
So we have that the number of variations of the  Sturm sequence of $g(x)$ at $y=-1/2$ is
\[ \omega(-1/2)=(-, -,  + , *),\]
the same of the  Sturm sequence of $g(x)$ at $y=0$.
Since  $ \omega(0) - \omega(-1/2) =0$, Sturm's theorem (Lemma \ref{lem:Sturm}), assures that there is no root of $P_G(x)$ in the interval $[-1/2, 0]$.
So, Fact (III) is true and, following the argumentation exposed before, Theorem \ref{the:main} is proved.
\medskip
\end{proof}
\smallskip

 We may recall  that, since  $ P_3$ is a induced subgraph of $K_1 \oplus (K_1 \cup K_2)$, in view of Remark \ref{rem:interl}, we could disregard the presence of copies of $P_3$ in the description of the relaxed block star considered to prove the main result of this subsection. However, it remains to consider the factor that appears in the factorization of $P_G(x)$ in the presence of two or more copies of the graph $P_3$.

 \begin{proposition}\label{prop:P3} If $G=K_1 \oplus \bigg( (p_1K_{r_1} \cup \ldots \cup p_kK_{r_k}) \cup q\left(K_1 \oplus (K_1 \cup K_2)\right) \cup s\ P_3)\bigg)$, where $r,p,q \geq 1$ and $s\geq 2$ then  $(x^2+4x+2)^{s-1}$ divides the characteristic polynomial $P_G(x)$
 of $G$. Furthermore, the roots of $(x^2+4x+2)$ are less than $-1/2$.
\end{proposition}

\begin{proof} The proof follows the same reasoning used in the proof of the lemma \ref{fator3}.
    Here, we  consider the following labeling:
let \( w_0 \) denote the universal vertex; we assign labels  \( w_1, \dots, w_N\)  to the vertices of the subgraph $ (p_1K_{r_1} \cup \ldots \cup p_kK_{r_k}) \cup q\left(K_1 \oplus (K_1 \cup K_2)\right)$, where $N={p_1r_1+p_2r_2+\dots+p_kr_k+4q}.$
We now label the vertices of each copy of $P_3$ as follows: let \( v_1, v_2, v_3 \) be the vertices of the first $P_3$ and  we repeat this labeling for each of the \( s \) copies. Specifically, for each \( i = 1, \dots, s \), the vertices \( v_{3(i-1)+1}, v_{3(i-1)+2}, v_{3(i-1)+3} \) correspond to the \( i \)-th copy of $P_3$.
Using this labeling, the matrix $\mathbf{D}(G)$ can be written as:
\[
\mathbf{D}(G)=\left[
\begin{array}{c|c}
\mathbf{D}_1&\mathbf{D}_2\\
\hline \
\mathbf{D}_2^{t}&\mathbf{D}_3
\end{array}
\right],
\]
where $\mathbf{D}_1$ has order $(N+1)\times (N+1)$, $\mathbf{D}_2$ has order $(N+1)\times 3s$ and finally $\mathbf{D}_3$ has order $3s\times 3s.$
The matrix $\mathbf{D}_2$  has the form
\[
\mathbf{D}_2=\left[
\begin{array}{cccc}
1&1&\dots&1\\
2&2&\dots&2\\
\vdots&\vdots&\ddots&\vdots\\
2&2&\dots&2\\
\end{array}
\right].
\]
Since the matrix \( \mathbf{D}_3 \) encodes the pairwise distances between the vertices \(v_1, \dots, v_{3s}\), it
 can naturally be viewed as a block matrix, where each block corresponds to the distances between specific subsets of vertices.
In fact,
\[
\mathbf{D}_3=\left[
\begin{array}{ccc|ccc|c|ccc}
0&1&2&2&2&2&\dots&2&2&2\\
1&0&1&2&2&2&\dots&2&2&2\\
2&1&0&2&2&2&\dots&2&2&2\\
\hline\
2&2&2&0&1&2&\dots&2&2&2\\
2&2&2&1&0&1&\dots&2&2&2\\
2&2&2&2&1&0&\dots&2&2&2\\
\hline\
&\dots&&&\dots&&\dots&&\dots\\
\hline\
2&2&2&2&2&2&\dots&0&1&2\\
2&2&2&2&2&2&\dots&1&0&1\\
2&2&2&2&2&2&\dots&2&1&0\\
\end{array}
\right].
\]
Then, like before, it can be proven that the eigenvalues of the matrix
 $$\mathbf{M}=\left[\begin{array}{cc}
-2&-1\\
-2&-2
 \end{array}\right]$$
 are eigenvalues of $\mathbf{D}(G)$.
 Given that the characteristic polynomial of \( \mathbf{M} \) is \( P_M(x) = x^2 + 4x + 2 \), and that all of its roots, which are less than $-1/2$, appear as eigenvalues of \( \mathbf{D}(G) \) with multiplicity at least \( s- 1 \), the assertion follows immediately.
\end{proof}

Thus, from Lemma \ref{lem:block-graph}, Theorems \ref{theo:pto-mvs2} and \ref{the:main} and  Proposition \ref{prop:P3},  we have the following  general result.

\begin{corollary}\label{cor:caract-diam2}
   A  graph  of diameter 2  satisfies the condition $\lambda_2 <-1/2$ if and only if it is a connected induced subgraph  of a relaxed block star graph as  described in (\ref{eq:geral-relax}).
\end{corollary}

\section*{Any triple block graph $Pt_2(p,q)$ satisfies $\lambda_2<-1/2$}

To complete the characterization, given in Theorem \ref{theo:pto-mvs3}, of graphs of diameter 3 which are not block graphs and have  $\lambda_2 < -1/2$, it remains to prove that the $Pt_1$ graph and any triple block graph $Pt_2(p,q)$, $p, q\geq 2$ ( recall Figure \ref{fig:PT1-PT2}) satisfy this spectral condition. The first fact follows straightforwardly, since $\lambda_2(Pt_1) \approx -0.50228$.
For the second,  considering that any $Pt_2(p,q)$ is a connected induced subgraph of $Pt_2(r,r)$ for $r=\max\{p,q\}$ and  taking into account Remark \ref{rem:interl}, it suffices to prove  the following result.

\begin{theorem}\label{the:Pt2}
Let $r\geq 2$  be an integer and $G=G_r$ be the triple block graph $Pt_2{(r,r)}$. Then the second largest eigenvalue of matrix $\mathbf{D}(G_r)$ is less than  $-1/2$.
\end{theorem}

\begin{proof}
For a fixed integer $r\geq 2$, let $G_r$ have its $2r+5$ vertices  labeled as indicated  in Figure \ref{fig:PT2(r,r)}.
\begin{figure}[h]
\begin{center}
\begin{tikzpicture} [scale=.35,auto=left]
 \tikzstyle{every node}=[circle, draw, fill=black,
                        inner sep=1.8pt, minimum width=4pt]

\node (o1) at (29,2) {};
\node at (29.5,1) [draw=none,fill=none] {$v_3$};
\node (p1) at (26,5)  {};
\node at (25.5,6) [draw=none,fill=none] {$v_4$};
\node (q1) at (26,2)  {};
\node at (25.5,1) [draw=none,fill=none] {$v_2$};
\node (r1) at (29,5) {};
\node at (29.5,6) [draw=none,fill=none] {$v_5$};
 \node (t1) at (27.5,8) {};
\node at (27.5,9) [draw=none,fill=none] {$v_1$};
 \node (s1) at (23,3) {};
\node (s2) at (23,6) {};
\node (y1) at (32,3) {};
\node (z1) at (32,6) {};

\foreach \from/\to in {o1/r1,o1/q1,o1/p1,p1/r1,p1/q1,q1/r1,s1/p1,s2/p1,r1/y1,r1/z1,t1/p1,t1/r1}
\draw (\from) -- (\to);

\draw[dashed] (s1) -- (s2) ;
 \draw[dashed] (y1) -- (z1) ;
\node at (24,4.5) [draw=none,fill=none] {$^{K_r}$};
\node at (31,4.5) [draw=none,fill=none] {$^{K_r}$};

\node at (28,-1) [draw=none,fill=none] {$Pt_2(p,q)$};

     \end{tikzpicture}
\caption{The graph  triple block $Pt_2(r,r)$. }
 \label{fig:PT2(r,r)}
\end{center}
\end{figure}
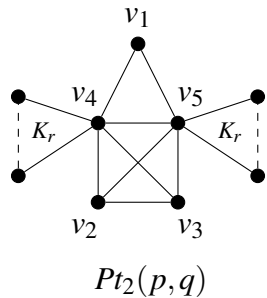

This labeling allows us to establish the partition $\pi = \{V_1, V_2, V_3, V_4\}$ of the vertex set of $G_r$, where:
 $V_1=\{v_1\}$\,, \
$V_2=\{v_2, v_3\}$\,, \
$V_3=\{v_4, v_5\}$\,, \ and
 $V_4=\left\{v_6, \dots, v_{r+5}, v_{r+6},\right.$ $\left. \ldots,  v_{2r+5}\right\}$, where $v_6, \dots, v_{r+5}, v_{r+6}, \ldots, v_{2r+5} $ are the labels assigned to the vertices of the two cliques $K_{r-1}$. Also,  accordingly to this labeling, the distance matrix $\mathbf{D}=\mathbf{D}(G_r)$
 can be described as
$$
\mathbf{D}=\left[\begin{array}{c|cc|cc|cccccccc}
0&2&2&1&1&2&2&\dots&2&2&2&\dots&2\\
\hline
2&0&1&1&1&2&2&\dots&2&2&2&\dots&2\\
2&1&0&1&1&2&2&\dots&2&2&2&\dots&2\\
\hline
1&1&1&0&1&1&1&\dots&1&2&2&\dots&2\\
1&1&1&1&0&2&2&\dots&2&1&1&\dots&1\\
\hline
2&2&2&1&2&0&1&\dots&1&3&3&\dots&3\\
2&2&2&1&2&1&0&\dots&1&3&3&\dots&3\\
\vdots&\vdots&\vdots&\vdots&\vdots&
\vdots&\vdots&\dots&\vdots&\vdots&\vdots&\dots&\vdots\\
2&2&2&1&2&1&1&\dots&0&3&3&\dots&3\\
\vdots&\vdots&\vdots&\vdots&\vdots&
\vdots&\vdots&\dots&\vdots&\vdots&\vdots&\dots&\vdots\\
2&2&2&2&1&3&3&\dots&3&0&1&\dots&1\\
2&2&2&2&1&3&3&\dots&3&1&0&\dots&1\\
\vdots&\vdots&\vdots&\vdots&\vdots&
\vdots&\vdots&\dots&\vdots&\vdots&\vdots&\dots&\vdots\\
2&2&2&2&1&3&3&\dots&3&1&1&\dots&0\\
\end{array}
\right]
$$
Furthermore,  $\pi$ is a distance equitable partition of the vertex set of $G_r$  having associated  distance divisor  matrix
\[\mathbf{H}_\pi(G_r)=\begin{bmatrix}
0 & 4 & 2 & 4r \\
2 & 1 & 2 &  4r\\
1 & 2 & 1 & 3r \\
2 & 4 & 3 & 4r-1 \\
\end{bmatrix}.\]
The characteristic polynomial of $\mathbf{H}_\pi(G_r)$ is given by
 \begin{equation} \label{eq:polPt2}P_{H_\pi}(x)={{x}^{4}}-\left( 4 r+1\right) \, {{x}^{3}}-\left( 25 r+15\right) \, {{x}^{2}}-\left( 43 r+19\right)  x-(16 r+6)\,.\end{equation}
We claim that the
 characteristic polynomial  $P_D(x)$ of $\mathbf{D}$ can be factorized as
   \begin{equation}\label{eq:fact} P_D(x)=(x+1)^{2r-1}\,(x^2+2(r+1)x+(r+1))\,P_{H_\pi}(x).
    \end{equation}
   To prove our claim we begin by observing  that  $\{v_6, v_7.,\ldots, v_{r+5}\}$ and $\{v_{r+6},\ldots, v_{2r+5}\}$ constitute two copies of $K_{r}$ and $\{v_2, v_3\}$ forms a $K_2$;  then, that  $-1$ is an eigenvalue of $\mathbf{D}$ with multiplicity at least $2r-1$ follows from Lemma \ref{lem:lemasLuHH}(i).  On the other hand, we know from Lemma \ref{lem:equit}, that $P_{H_\pi}(x)$ divides $P_D(x)$.
Also, we can guarantee that the polynomial $(x^2+2(r+1)x+(r+1))$\  divides $P_D(X).$
 The proof follows directly from the fact that
    the eigenvalues of the  matrix
 $$\mathbf{M}=\begin{bmatrix}
-1&-r\\
-1&-2r-1
 \end{bmatrix}$$
 are eigenvalues of $\mathbf{D}=\mathbf{D}(G_r).$
  Indeed,  let \( \lambda \) be an eigenvalue of the matrix \( \mathbf{M} \), and let \( [x_\lambda, y_\lambda]^{t} \) denote an eigenvector associated with \( \lambda \).
Then the  vector
$$
    \mathbf{u}=[0, 0, 0, x_\lambda, -x_\lambda,\underbrace{y_\lambda, \dots,  y_\lambda}_{r-1}, \underbrace{-y_\lambda, \dots,  -y_\lambda}_{r-1}]^{t}
    $$
is an eigenvector of $\mathbf{D}$ associated with $\lambda.$
Given that  \( P_M(x) = x^2 + 2(r+1)x + (r+1) \), and that its roots appear as eigenvalues of \( \mathbf{D} \), we conclude that $P_M(x)$ divides $P_D(x)$ as asserted. Our claim is proved.

To finalize the proof of  Theorem \ref{the:Pt2}, we must now turn our attention to the behavior of the  eigenvalues of $\mathbf{D}$ regarding to $-1/2$.
 First of all, the roots of $P_M(x)=x^2+2(r+1)+(r+1)$ are  strictly less than $-1/2.$
In fact, the roots are $x_1=-(r+1)-\sqrt{r^2+r}<-(r+1) <-r\leq -1$ and $x_2=-(r+1)+\sqrt{r^2+r}$;
but $x_2 < -1/2$ if and only if $\sqrt{r^2+r} < r+1/2$, or equivalently, $ 0 < 1/4$, which is true, so $x_2<-1/2$.

It remains to analyze the roots of the polynomial $P_{H_{\pi}}(x)$. 
Firstly, from (\ref{eq:polPt2}),  we may note that all the coefficients of
$P_{H_{\pi}}(x)$, except the leader one,   are negative, due the conditions on $r$.
 Thus, the sequence of signs between consecutive coefficients of $P_{H_\pi}(x)$  is
\[(+,-,-,-,-). \]
Noting that  $P_{H_\pi}(x)$  presents only one sign changes, Descartes' Rule of signs (Lemma \ref{lem:Desc}) says that
\[\#_+(P_{H_\pi}(x))=\nu(P_{H_\pi}(x)). \]
Since  we know that $\lambda_1(G_r) >0$ is an eigenvalue of $\mathbf{H}_{\pi}(G_r)$, we may conclude that   $P_{H_\pi}(x))$ has only one non-negative root.
 Now, we may note that
\[
0 > P_{H_\pi}\left(-\tfrac{1}{2}\right) = \frac{-(4r+1)}{16} > -6 - 16r =P_{H_\pi}(0).
\]
Furthermore, the second derivative of \( P_{H_\pi} \) is given by
\[P_{H_\pi}''(x)=12 {{x}^{2}}-6 \left( 4 r+1\right)  x-2 \left( 25 r+15\right)\,, \]
which satisfies \( P_{H_\pi}''(x) < 0 \) for all \( x \in \left[-\tfrac{1}{2}, 0\right] \). Therefore \( P_{H_\pi} \) is strictly concave on this interval. Since the derivative $P_{H_\pi}'(x)$ strictly decreases and  $P_{H_\pi}'(-1/2) =-21r-21/4$, then it does not vanish in this interval. Then the maximum of \( P_{H_{\pi}}(x) \) on the interval is attained at one of the endpoints. For all these reasons, it follows  that $P_{H_\pi}$ has no roots in the interval $\left[-\tfrac{1}{2}, 0\right]$.
Thus, considering the factorization of $P_D(x)$ given in (\ref{eq:fact}) and the facts  just proved, we can conclude that the second largest  eigenvalue of $\mathbf{D}(G_r)$ is less than $-1/2$, just as we wanted.
\end{proof}

  This way, joining Theorems \ref{theo:pto-mvs3} and \ref{the:Pt2}, we obtain the following result.

  \begin{corollary}\label{cor:caract-diam3}
      A graph  of diameter 3 which is not a block graph,  satisfies the condition $\lambda_2 <-1/2$ if and only if it is a connected induced subgraph  of $Pt_1$ graph or of any triple block graph $Pt_{(p,q)}$.
  \end{corollary}

\section*{Acknowledgments}
The authors thank  Pedro Henrique Duarte Santos for the description of  relaxed block star graphs, which was obtained with the aid of the  GraphFilter \cite{graphfilter}.


\bibliographystyle{acm}
\bibliography{biblio-D}   

\begin{thebibliography}{10}

\bibitem{Anderson01051998}
{\sc Anderson, B., Jackson, J., and Sitharam, M.}
\newblock Descartes' rule of signs revisited.
\newblock {\em The American Mathematical Monthly 105}, 5 (1998), 447--451.

\bibitem{AOUCHICHE2014301}
{\sc Aouchiche, M., and Hansen, P.}
\newblock Distance spectra of graphs: A survey.
\newblock {\em Linear Algebra and its Applications 458\/} (2014), 301--386.

\bibitem{Beineke_Golumbic_Wilson_2021}
{\sc Beineke, L.~W., Golumbic, M.~C., and Wilson, R.~J.}, Eds.
\newblock {\em Topics in algorithmic graph theory}, vol.~178 of {\em
  Encyclopedia of Mathematics and its Applications}.
\newblock Cambridge University Press, Cambridge, 2021.

\bibitem{biagioli2016methods}
{\sc Biagioli, E.~J.}
\newblock Methods for bounding and isolating the real roots of univariate
  polynomials.
\newblock {\em IMPA\/} (2016).

\bibitem{BP93}
{\sc Blair, J. R.~S., and Peyton, B.}
\newblock An introduction to chordal graphs and clique trees.
\newblock In {\em Graph theory and sparse matrix computation}, vol.~56 of {\em
  IMA Vol. Math. Appl.} Springer, New York, 1993, pp.~1--29.

\bibitem{CRS2010}
{\sc Cvetkovi\'{c}, D., Rowlinson, P., and Simi\'{c}, S.}
\newblock {\em An introduction to the theory of graph spectra}, vol.~75 of {\em
  London Mathematical Society Student Texts}.
\newblock Cambridge University Press, Cambridge, 2010.

\bibitem{Di61}
{\sc Dirac, G.~A.}
\newblock On rigid circuit graphs.
\newblock {\em Abh. Math. Sem. Univ. Hamburg 25\/} (1961), 71--76.

\bibitem{Go04}
{\sc Golumbic, M.~C.}
\newblock {\em Algorithmic graph theory and perfect graphs}, second~ed.,
  vol.~57 of {\em Annals of Discrete Mathematics}.
\newblock Elsevier Science B.V., Amsterdam, 2004.
\newblock With a foreword by Claude Berge.

\bibitem{guo2024graphs}
{\sc Guo, H., and Zhou, B.}
\newblock Graphs for which the second largest distance eigenvalue is less than-
  12.
\newblock {\em Discrete Mathematics 347}, 9 (2024), 114082.

\bibitem{Ho81}
{\sc Howorka, E.}
\newblock A characterization of {P}tolemaic graphs.
\newblock {\em J. Graph Theory 5}, 3 (1981), 323--331.

\bibitem{graphfilter}
{\sc Jones, A.~A., de~Castro, L.~B., Pimenta, F.~S., and Pinto, I. R.~F.}
\newblock Graph filter: software for manipulating and searching graphs. version
  13.3.9, july 25, 2023.
\newblock Available on \url{http://sistemas.jf.ifsudestemg.edu.br/graphfilter}.
\newblock Accessed: 12-01-2025.

\bibitem{SurveyLinShu}
{\sc Lin, H., Shu, J., Xue, J., and Zhang, Y.}
\newblock A survey on distance spectra of graphs.
\newblock {\em Adv. Math. (China) 50}, 1 (2021), 29--76.

\bibitem{Liu-Xue-Guo}
{\sc Liu, R., Xue, J., and Guo, L.}
\newblock On the second largest distance eigenvalue of a graph.
\newblock {\em Linear Multilinear Algebra 65}, 5 (2017), 1011--1021.

\bibitem{LuHuangHuang2017}
{\sc Lu, L., Huang, Q., and Huang, X.}
\newblock The graphs with exactly two distance eigenvalues different from
  {$-1$} and {$-3$}.
\newblock {\em J. Algebraic Combin. 45}, 2 (2017), 629--647.

\bibitem{MP10}
{\sc Markenzon, L., and da~Costa~Pereira, P.~R.}
\newblock One-phase algorithm for the determination of minimal vertex
  separators of chordal graphs.
\newblock {\em Int. Trans. Oper. Res. 17}, 6 (2010), 683--690.

\bibitem{Wang2004}
{\sc Wang, X.}
\newblock A simple proof of {D}escartes's rule of signs.
\newblock {\em The American Mathematical Monthly 111}, 6 (2004), 525--526.

\bibitem{Xing-Zhou16}
{\sc Xing, R., and Zhou, B.}
\newblock On the second largest distance eigenvalue.
\newblock {\em Linear Multilinear Algebra 64}, 9 (2016), 1887--1898.

\bibitem{Xue-Lin-Shu-block}
{\sc Xue, J., Lin, H., and Shu, J.}
\newblock On the second largest distance eigenvalue of a block graph.
\newblock {\em Linear Algebra Appl. 591\/} (2020), 284--298.

\bibitem{YanChen1996}
{\sc Yan, J.-H., Chen, J.-J., and Chang, G.~J.}
\newblock \emph{Quasi}-threshold graphs.
\newblock {\em Discrete Appl. Math. 69}, 3 (1996), 247--255.

\bibitem{yang2025tricyclicgraphssecondlargest}
{\sc Yang, K., and Wang, L.}
\newblock Tricyclic graphs for which the second largest distance eigenvalue
  less than $-1/2$.
\newblock Available on \url{https://arxiv.org/abs/2509.12640h}, 2025.
\newblock preprint.

\end{thebibliography}
\end{document}